\documentclass[a4paper,bibliography=totoc,12pt]{scrartcl}

\usepackage[english]{babel}
\usepackage[utf8]{inputenc}
\usepackage[T1]{fontenc}
\usepackage{lmodern}
\usepackage{textcomp}
\usepackage{amsfonts}
\usepackage{mathtools}
\usepackage{amssymb}
\usepackage[thmmarks,hyperref,amsmath]{ntheorem}
\usepackage{needspace}
\usepackage{csquotes}
\usepackage{siunitx}
\usepackage[shortlabels]{enumitem}
\usepackage[pdftex]{hyperref}
\usepackage{url}
\usepackage{xcolor}

\DeclareFontEncoding{LS1}{}{}
\DeclareFontSubstitution{LS1}{stix}{m}{n}
\DeclareSymbolFont{symbols2}{LS1}{stixfrak} {m} {n}
\DeclareMathSymbol{\operp}{\mathbin}{symbols2}{"A8}

\usepackage[
	backend=biber,
	style=numeric,
	giveninits = true,
	sorting = nyt,
	maxnames = 8,
	sortlocale = en_US,
	doi=true,
	isbn=false,
	date=year
]{biblatex}
\AtBeginBibliography{\small}
\DeclareFieldFormat{doi}{%
	\newline
	\mkbibacro{DOI}\addcolon\space
	\ifhyperref
		{\href{http://dx.doi.org/#1}{\nolinkurl{#1}}}
		{\nolinkurl{#1}}}
\DeclareFieldFormat{url}{%
	\newline
	\mkbibacro{URL}\addcolon\space
	\ifhyperref
		{\href{#1}{#1}}
		{\nolinkurl{#1}}}

\addtokomafont{disposition}{\rmfamily}
\addtokomafont{title}{}

\setlist[description]{
	font={\rmfamily},
}
\setlist{itemsep=2pt}
\setlist{parsep=3pt}

\bibliography{bib_abrv.bib, degree_d_functions.bib}

\newcommand{\F}{\mathbb{F}}
\newcommand{\R}{\mathbb{R}}

\newcommand{\N}{\mathbb{N}}
\newcommand{\Z}{\mathbb{Z}}
\newcommand{\Q}{\mathbb{Q}}
\newcommand{\vek}[1]{\boldsymbol{#1}}
\newcommand{\qbinom}[3]{\genfrac{[}{]}{0pt}{}{#1}{#2}_{#3}}
\newcommand{\qbinomS}[2]{\genfrac{[}{]}{0pt}{}{#1}{#2}}
\newcommand{\qnumb}[2]{[#1]_{#2}}
\newcommand{\qnumbS}[1]{[#1]}

\newcommand{\powerset}{\mathfrak{P}}

\DeclareMathOperator{\PGammaL}{P\Gamma L}

\DeclareMathOperator{\supp}{supp}
\DeclareMathOperator{\Aut}{Aut}
\DeclareMathOperator{\im}{im}

\DeclareMathOperator{\PG}{PG}

\DeclareMathOperator{\rowsp}{rowsp}
\DeclareMathOperator{\rk}{rk}
\DeclareMathOperator{\cork}{cork}

\DeclareMathOperator{\wt}{wt}
\DeclareMathOperator{\Der}{Der}
\DeclareMathOperator{\Res}{Res}

\theoremheaderfont{\bfseries}
\theorembodyfont{\upshape}
\theoremstyle{plain}
\theoremseparator{.}
\theoremsymbol{\ensuremath{\diamond}}
\newtheorem{theorem}{Theorem}[section]
\newtheorem{lemma}[theorem]{Lemma}
\newtheorem{fact}[theorem]{Fact}
\newtheorem{corollary}[theorem]{Corollary}

\newtheorem{definition}[theorem]{Definition}
\newtheorem{remark}[theorem]{Remark}

\newtheorem{example}[theorem]{Example}

\theoremheaderfont{\itshape}
\theorembodyfont{\upshape}
\theoremstyle{nonumberplain}
\theoremseparator{.}
\theoremsymbol{\rule{1.2ex}{1.2ex}}

\newtheorem{proof}{Proof}

\begin{document}
\title{The degree of functions\\in the Johnson and $q$-Johnson schemes}
\author{
Michael Kiermaier%
\thanks{
Universität Bayreuth, Institute for Mathematics, 95440 Bayreuth, Germany
\newline
email:~\texttt{michael.kiermaier@uni-bayreuth.de}
\newline
homepage:~\url{https://mathe2.uni-bayreuth.de/michaelk/}
}
\and
Jonathan Mannaert%
\thanks{
Vrije Universiteit Brussel, Department of Mathematics and Data Science, Pleinlaan 2, B--1050 Brussels,
Belgium
\newline
email:~\texttt{Jonathan.Mannaert@vub.be}
}
\and
Alfred Wassermann%
\thanks{
Universität Bayreuth, Institute for Mathematics, 95440 Bayreuth, Germany
\newline
email:~\texttt{alfred.wassermann@uni-bayreuth.de}
}
}
\maketitle

\begin{abstract}
In 1982, Cameron and Liebler investigated certain \emph{special sets of lines} in $\PG(3,q)$, and gave several equivalent characterizations.
Due to their interesting geometric and algebraic properties, these \emph{Cameron-Liebler line classes} got much attention.
Several generalizations and variants have been considered in the literature, the main directions being a variation of the dimensions of the involved spaces, and studying the analogous situation in the subset lattice.
An important tool is the interpretation of the objects as Boolean functions in the \emph{Johnson} and \emph{$q$-Johnson schemes}.

In this article, we develop a unified theory covering all these variations.
Generalized versions of algebraic and geometric properties will be investigated, having a parallel in the notion of \emph{designs} and \emph{antidesigns} in association schemes, which is connected to Delsarte's concept of \emph{design-orthogonality}.
This leads to a natural definition of the \emph{degree} and the \emph{weights} of functions in the ambient scheme, refining the existing definitions.
We will study the effect of dualization and of elementary modifications of the ambient space on the degree and the weights.
Moreover, a divisibility property of the sizes of Boolean functions of degree $t$ will be proven.
\end{abstract}

\section{Introduction}
In~\cite{Cameron-Liebler-1982-LAA46:91-102}, Cameron and Liebler worked on the classification of irreducible collineation subgroups of $\PG(3,q)$, having equally many point and line orbits.
In order to do so, they showed that the line orbits admit certain equivalent properties, earning them the name \emph{special line class} $\mathcal{L}$.
An algebraic property says that the characteristic function of $\mathcal{L}$ lies in the row space of the point-line incidence matrix of $\PG(3,q)$.
Surprisingly, this is equivalent to the following property more of a geometric nature, stating that $\mathcal{L}$ has constant intersection size $x$ with any line spread, i.e. with any partition of the point set into a set of lines.
The value $x$ is known as the \emph{parameter} of the special line class and can also be obtained from $\
\#\mathcal{L} = x(q^2+q+1)$.

Cameron and Liebler conjectured that all irreducible collineation groups of $\PG(3,q)$ with more than one orbit have either a fixed point or just two point orbits, one of which is a plane.
This conjecture implies that every special line class of $\PG(3,q)$ is trivial.
More precisely, this means that up to taking complements, it is either empty, or a point-pencil, or the set of all lines in a plane, or the disjoint union of the latter two.

\subsection{Classical Cameron-Liebler line classes}
In the literature, these special line classes have soon been called \emph{Cameron-Liebler line classes} \cite{Penttila-1991-GeomDed37[3]:245-252}.
The conjecture of Cameron and Liebler was first disproven by Drudge in \cite{DRUDGE1999263}, where he constructed a non-trivial Cameron-Liebler line class of parameter $x=5$ in $\PG(3,3)$.
This eventually led to an infinite family of examples of parameter $x=\frac{q^2+1}{2}$ in $\PG(3,q)$ for $q$ odd, see \cite{BRUEN199935}.
These results immediately drew much interest to this question, with the main focus on finding new infinite families of non-trivial examples of Cameron-Liebler line classes, or giving non-existence conditions on the parameter $x$.

Some major non-existence results can be found in \cite{Filmus-Ihringer-2019-JCTSA162:241-270,GAVRILYUK2014224,MR3115334}.
While many other results exist, these three references give a good understanding on the situation today.
In $\PG(3,q)$ many other non-trivial infinite families of examples have been found.
Non-isomorphic examples with the same parameter $x = \frac{q^2+1}{2}$ as in the initial construction~\cite{BRUEN199935} have been obtained for $q \equiv 1 \bmod 4$ in \cite{COSSIDENTE2019104}, for $q\geq 7$ in \cite{MR3927192}, and for $q\geq 5$ in \cite{MR3742843}.
A family with a different parameter $x=\frac{q^2-1}{2}$ has been constructed in \cite{DBDMKR2016,FENG2015307} for $q \equiv 5,9 \bmod{12}$.
Finally, an example of parameter $x=\frac{(q+1)^2}{3}$ for $q \equiv 2\bmod 3$ was described in \cite{FENG2021107780}.
To the best of our knowledge, all currently known constructions are listed above.

\subsection{Generalizations}

Many generalizations of the notion of Cameron-Liebler line classes have been studied.

\paragraph{Generalization of the dimension $n$ of the ambient space.}
While the setting in the original paper by Cameron and Liebler \cite{Cameron-Liebler-1982-LAA46:91-102} was already a general ambient geometry $\PG(n-1,q)$, the notion of special classes of lines had been restricted to $\PG(3,q)$.
The study of Cameron-Liebler line classes in the general $\PG(n-1,q)$ was first conducted in Drudge's thesis \cite{ThesisDrudge}.

\paragraph{Generalization of the dimension $k$ of the subspaces.}
A second generalization was described in \cite{Rodgers-Storme-Vansweevelt-2018-Comb38[3]:739-757}, where instead of looking at sets of lines, the authors look at sets of $k$-spaces in $\PG(2k-1,q)$.%
\footnote{In this article, we will always use vector space dimension, except in established symbols like $\PG$.
Hence $k$-spaces are spaces of geometric dimension $k-1$.}
This approach was combined with generalized dimension $n$ in \cite{Blokhuis-DeBoeck-DHaeseleer-2019-DCC87[8]:1839-1856}, resulting in objects called \emph{Cameron-Liebler sets of $k$-spaces}.

All the generalizations so far are based on the natural adaption of the above stated algebraic characterization of Cameron-Liebler line classes, resulting in the investigation of characteristic vectors in the row space of the point-($k$-space) incidence matrix in $\PG(n-1,q)$.
In \cite{Filmus-Ihringer-2019-JCTSA162:241-270}, a connection with \emph{Boolean degree $1$ functions} in the Johnson or $q$-Johnson scheme was observed and described, leading to a full classification for $n-k\geq 2$, where $q \leq 5$ and $n\geq 5$, or, $q=2$ and $n=4$.

In this generalized situation, again a \emph{parameter} $x$ is defined, but unlike for the classical Cameron-Liebler line classes in $\PG(3,q)$, the number $x$ may be a proper fraction in the case that $k \nmid n$.
Strikingly, no non-trivial examples of Cameron-Liebler sets of $k$-spaces have been found so far, except the classical examples in $\PG(3,q)$.
Hence there has been much effort on non-existence results.
Most of these results rule out certain parameters.
The most recent development is found in \cite{DEBEULE2022102047, DeBeule-Mannaert-Storme-2024-DCC-prelim, Filmus-Ihringer-2019-JCTSA162:241-270, ihringer2024classification}, where the last reference shows that for fixed $k$ and $n$ sufficiently large, only trivial examples do exist.

\paragraph{The set case $q=1$.}
Via the interpretation as $q$-analogs, there is a natural counterpart of the definitions in the case $q=1$.
The counterpart of Cameron-Liebler sets of $k$-spaces are Boolean degree $1$ functions, which have been shown to be trivial in all cases, see~\cite{Filmus-Ihringer-2019-JCTSA162:241-270}.

\paragraph{Generalization of the degree $t$.}
The approach as Boolean functions opens the way for yet another direction of generalization, namely of the degree $t$.
In the algebraic characterization, this corresponds to the replacement of the points by general $t$-spaces.
While the degree $t=1$ case has been studied extensively, this is not the case for higher degree $t$.
For $t=2$, some examples are given in \cite{DeBeule-DHaeseleer-Ihringer-Mannaert-2023-ElecJComb30[1]:P1.31}.
Here, the authors also explain the implications of Delsarte's concept of design-orthogonality \cite{Delsarte-1977-PhilRR32[5-6]:373-411} for degree $2$, see also \cite{Roos-1982-DelftProgrRep7[2]:98-109}.
Moreover, in \cite{FILMUS,Filmus-Ihringer-2019-DM342[12]:111614}, it has been shown that in the set case $q=1$ there exists a constant $m(t)$ such that for all $2t \leq k \leq \frac{n}{2}$, any Boolean degree $t$ function is an $m(t)$-junta, which intuitively means that it depends only on at most $m(t)$ coordinates of the underlying base set.
Moreover, the bound $2t \leq k$ has been shown to be sharp~\cite{FILMUS}.

\paragraph{A unified generalization: Antidesigns.}
The goal of this paper is to develop a coherent theory covering all the above generalizations.
In our investigations, the connection to the theory of designs turned out to be fruitful at several places.
Actually, degree $t$ functions are the same as $t$-antidesigns in the Johnson (for $q = 1$) or $q$-Johnson (for $q \geq 2$) scheme, which have been investigated already in \cite{Roos-1982-DelftProgrRep7[2]:98-109} in 1982, the same year as the original paper of Cameron and Liebler \cite{Cameron-Liebler-1982-LAA46:91-102}.

We remark that our definition of the degree in Section~\ref{sec:weights_degree} is based on the chain of row spaces of the incidence matrices of $t$-subspaces vs. $k$-subspaces.
This chain is equivalent to the decomposition of the ($q$-)Johnson scheme into orthogonal eigenspaces, see Remark~\ref{rem:vi_vs_johnson} for more details.
While the theory of association schemes provides powerful algebraic methods, it turned out that these are not required for our results and therefore we decided to use this pure combinatorial approach.
The choice of the title of this article connects to the existing literature on (generalized) Cameron-Liebler sets, which are usually based on the $q$-Johnson scheme.

\subsection{Structure of the paper}
In Section~\ref{sec:prelim} the required preliminaries are given.
Lemma~\ref{lem:lambda_min} contains a formula on the parameter $\lambda_{\min}$ from design theory which, to the best of our knowledge, is new.
Section~\ref{sec:alg_geom} investigates algebraic and geometric properties, which -- as we will see -- can also be understood as antidesigns and designs.
The main result of that section is the characterization in Theorem~\ref{thm:alg_geom}.
In Section~\ref{sec:weights_degree}, \emph{degree} and \emph{weights} are defined, and fundamental properties are derived.
Theorem~\ref{thm:size_divisibility} gives a divisibility property of the size of integer valued (and in particular, Boolean) functions of degree $t$.
In Section~\ref{sec:dual}, the effect of dualization on the degree (Theorem~\ref{thm:dual_v_space}) and on the weights (Theorem~\ref{thm:dual_wt_dist}) is studied.
Section~\ref{sec:weights_degree:ambient} investigates the implications of the two elementary ways to change the ambient space in Theorem~\ref{thm:change_ambient_P} and Theorem~\ref{thm:change_ambient_H}.
As an application, we get the degree of the \emph{basic sets} in Theorem~\ref{thm:deg_f}.
Finally, the case $q=1$, $n=6$ and $k=3$ will be discussed in more detail in Section~\ref{sec:example}.

\section{Preliminaries}\label{sec:prelim}
\subsection{$q$-numbers and $q$-binomial coefficients}
For an integer $n$ and a prime power $q \geq 2$, we define $\qnumb{n}{q} = \qnumbS{n} = \frac{q^n-1}{q-1}$.
For $n \geq 0$, this expression cancels to $\qnumb{n}{q} = 1 + q + \ldots + q^{n-1}$, and for all $n\in\Z$, one checks the negation formula $\qnumb{-n}{q} = -\frac{1}{q^n}\qnumb{n}{q}$.
Therefore, we may also allow $q = 1$, which yields $\qnumb{n}{1} = n$ and which is also the reason for the name \emph{$q$-analog of the number $n$} for $\qnumb{n}{q}$.
An important property (and further analogy) is $\gcd(\qnumb{a}{q}, \qnumb{b}{q}) = \qnumb{\gcd(a,b)}{q}$ for all $a,b\in\N$.%
\footnote{In this article, $\N = \{0,1,2,\ldots\}$.}

For a non-negative integer $k$, we use the notation of the \emph{falling $q$-factorial} $(\qnumb{n}{q})^{\underline{k}} = \prod_{i=0}^{k-1}\qnumb{n-i}{q}$ and the \emph{$q$-factorial} $\qnumb{k}{q}! = \qnumbS{k}! = (\qnumb{k}{q})^{\underline{k}}$.
Note that $\qnumb{k}{1}! = k!$.
Now for $n,k\in\Z$, the \emph{Gaussian binomial coefficient} or \emph{$q$-binomial coefficient} is defined as
\[
	\qbinom{n}{k}{q}
	= \qbinomS{n}{k}
	= \begin{cases}
		\frac{(\qnumb{n}{q})^{\underline{k}}}{\qnumb{k}{q}!} & \text{if } k \geq 0\text{,}\\
		0 & \text{if }k < 0\text{,}
	\end{cases}
\]
resembling the common binomial coefficient for $q = 1$.
For $n \geq 0$, the statement $\qbinom{n}{k}{q} \neq 0$ is equivalent to $k\in\{0,\ldots,n\}$, and in this case we may use the familiar formula
\[
	\qbinom{n}{k}{q} = \frac{\qnumb{n}{q}!}{\qnumb{k}{q}!\qnumb{n-k}{q}!}\text{.}
\]
For fixed non-negative integers $0 \leq k \leq n$, the Gaussian binomial coefficient $\qbinom{n}{k}{q}$ is a polynomial in $\Z[q]$ of degree $k(n-k)$.
By the $q$-analog of Kummer's theorem, its factorization into irreducibles given by
\begin{equation}\label{eq:q_kummer}
	\qbinom{n}{k}{q} = \prod_{\substack{i\in\Z_{>0} \\ (n \bmod i) < (k\bmod i)}} \Phi_i\text{,}
\end{equation}
where $\Phi_i$ denotes the $i$-th cyclotomic polynomial \cite[Eq.~(10)]{Knuth-Wilf-1989-JRAM396:212-219} and $(n \bmod k)$ denotes the remainder of the integer division of $n$ by $k$.

Based on $(\qnumb{-n}{q})^{\underline{k}} = (-1)^k \frac{1}{q^{kn + \binom{k}{2}}} (\qnumb{n+k-1}{q})^{\underline{k}}$ (where $k\geq 0$), we get the negation formula
\begin{equation}\label{eq:qbinom_negation}
	\qbinom{-n}{k}{q} = (-1)^k \frac{1}{q^{kn + \binom{k}{2}}} \qbinom{n+k-1}{k}{q}
\end{equation}
for all $n,k\in\Z$.

As basic properties of these numbers, we mention the reflection formula
\[
	\qbinom{n}{k}{q} = \qbinom{n}{n - k}{q} \qquad \text{for all }n,k\in\Z\text{ with }n \geq 0\text{,}
\]
and the $q$-Pascal triangle identities
\begin{equation*}
	\qbinom{n}{k}{q} = \qbinom{n-1}{k-1}{q} + q^k\qbinom{n-1}{k}{q} = q^{n-k}\qbinom{n-1}{k-1}{q} + \qbinom{n-1}{k}{q}
	\qquad\text{for all }n,k\in\Z\text{.}
\end{equation*}
An important formula is the $q$-Vandermonde identity
\begin{equation}\label{eq:q_vandermonde}
	\qbinom{n+m}{k}{q} = \sum_{i=0}^k \qbinom{m}{k-i}{q} \qbinom{n}{i}{q} q^{i(m-k+i)}
	\qquad\text{for all }n,m,k\in\Z\text{.}
\end{equation}

\subsection{Subset and subspace lattices}
Let $n$ be an integer and $q$ a prime power.
The value $q=1$ is allowed and called the \emph{set case}, where we fix a set $V$ of size $n$.
The case $q \geq 2$ is the \emph{$q$-analog case}, where we fix an $\F_q$ vector space of dimension $n$.
We denote the lattice of all subsets (for $q = 1$) or subspaces (for $q \geq 2$) by $\mathcal{L}(V)$.

We use the unified notation $\rk X = \#X$ for sets $X$ and $\rk X = \dim X$ for $\F_q$ vector spaces $X$.
The notation is based on the fact that it is the lattice-theoretic rank of $X$ in the graded lattice $\mathcal{L}(V)$ in both cases.
Moreover, we define the \emph{corank} $\cork X = n - \rk X$.
We will call an element of $U \in \mathcal{L}(V)$ a \emph{subspace} (or \emph{$(\rk U)$-subspace}), denoted by $U \leq V$ for a coherent treatment of both cases.
The set of all subspaces of $V$ of fixed rank $k\in\Z$ will be denoted by $\qbinom{V}{k}{q}$ or $\qbinomS{V}{k}$;
its cardinality is
\[
	\#\qbinom{V}{k}{q} = \qbinom{n}{k}{q}\text{.}
\]
The elements of $\qbinomS{V}{1}$, $\qbinomS{V}{2}$, $\qbinomS{V}{3}$ and $\qbinomS{V}{n-1}$ will be called \emph{points}, \emph{lines}, \emph{planes} and \emph{hyperplanes}, respectively.
There is a unique subspace of rank $0$, which is $\emptyset$ for $q=1$ and the zero space $\{\boldsymbol{0}\}$ for $q \geq 2$.
We will use the uniform notation $\boldsymbol{0}$.

In the set case, each $U \subseteq V$ has a unique complement $U^\complement \coloneqq V \setminus U$.
The automorphism group $\Aut(\mathcal{L}(V))$ equals the symmetric group $S_V$ on $V$.

For $q \geq 2$ and $n \geq 3$, the automorphism group $\Aut(\mathcal{L}(V))$ is given by $\PGammaL(V)$ by the fundamental theorem of projective geometry.%
\footnote{For $n = 2$, the group is larger: $\Aut(\mathcal{L}(V))$ is induced by the symmetric group on $\qbinomS{V}{1}$.}
The lattice $\mathcal{L}(V)$ is known to be self-dual, where an anti-automorphism given by $U \mapsto U^\complement = V\setminus U$ in the case $q=1$ and $U \mapsto U^\perp$, where $U^\perp$ is the orthogonal subspace with respect to some non-degenerate bilinear form on $V$.
In the following, $U \mapsto U^\perp$ will uniformly denote some fixed anti-automorphism of $\mathcal{L}(V)$.
For any $W\in\mathcal{L}(V)$, the map $\mathcal{L}(W) \to \mathcal{L}(V/W^\perp)$, $U \mapsto U^\perp / W^\perp$ is an anti-isomorphism of lattices.

For $\mathcal{F} \subseteq \mathcal{L}(V)$ and a chain $I \subseteq J \subseteq V$, we define
\[
	\mathcal{F}|_I^J = \{F\in \mathcal{F} \mid I \subseteq F \subseteq J\}.
\]
Moreover, we use the abbreviations
\[
	\mathcal{F}|_I = \mathcal{F}|_I^V
	\quad\text{and}\quad
	\mathcal{F}|^J = \mathcal{F}|_{\boldsymbol{0}}^J\text{.}
\]

We need the following well-known counting formula, see for example \cite[Lem.~1]{Braun-Kiermaier-Wassermann-2018-SignalsCommunTechnol:171-211} or \cite[Lem.~9.3.2]{Brouwer-Cohen-Neumaier-1989}.
\begin{lemma}\label{lem:subspace_counting}
	Let $a,b,u\in\{0,\ldots,n\}$ and let $B \in\qbinomS{V}{b}$.
	For $A\in\qbinomS{B}{a}$, we have
	\[
		\#\Big\{U \in\qbinomS{V}{u} \mid U \cap B = A\Big\} = q^{(b-a)(u-a)} \qbinom{n-b}{u-a}{q}
	\]
	and
	\[
		\#\Big\{U \in\qbinomS{V}{u} \mid \rk (U \cap B) = a\Big\} = q^{(b-a)(u-a)} \qbinom{b}{a}{q} \qbinom{n-b}{u-a}{q}\text{.}
	\]
\end{lemma}

The basic objects of our investigation will be the functions $f: \qbinomS{V}{k} \to \R$ for fixed dimension $k$.
We will silently identify the function $f$ with the vector $(f(K))_{K\in\qbinomS{V}{k}} \in \R^{\qbinomS{V}{k}}$.
In the case $\im f \subseteq \{0,1\}$, the function $f$ is called \emph{Boolean}.
A Boolean function $f$ can be identified with its support $\mathcal{F} = f^{-1}(\{1\}) \subseteq \qbinomS{V}{k}$.
In this way, we get a bijective correspondence between the set $\{0,1\}^{\qbinomS{V}{k}}$ of all Boolean functions on $\qbinomS{V}{k}$ and the power set $\powerset(\qbinomS{V}{k})$ of $\qbinomS{V}{k}$, where the reverse direction is given by the \emph{characteristic function}
\[
	\chi_\mathcal{F} : K \mapsto\begin{cases} 1 & \text{if }K\in \mathcal{F}\text{;}\\ 0 & \text{if }K\notin\mathcal{F}\text{.}\end{cases}
\]
The all-one function $\qbinomS{V}{k} \to \R$ will be denoted by $\boldsymbol{1}_{\qbinomS{V}{k}}$ or simply $\boldsymbol{1}$.
Clearly, $\boldsymbol{1}_{\qbinomS{V}{k}} = \chi_{\qbinomS{V}{k}}$.
For $U\in\mathcal{L}(V)$, we denote the characteristic function of the \emph{$(\rk U)$-pencil} $\qbinomS{V}{k}|_U$ by $\boldsymbol{x}_U^{(k)}$ and the characteristic function of the \emph{dual $(\cork U)$-pencil} $\qbinomS{V}{k}|^U$ by $\bar{\boldsymbol{x}}_U^{(k)}$.%
\footnote{Remember that $\cork U$ denotes the corank of $U$ in $V$, i.\,e. $\cork U = n - \rk U$.}
Note that $\boldsymbol{x}_U^{(k)} = \boldsymbol{0}$ for $\rk(U) > k$, and $\bar{\boldsymbol{x}}_U^{(k)} = \boldsymbol{0}$ for $\cork(U) > n-k$.
Moreover, $\boldsymbol{x}_{U_1}^{(k)} \cdot \boldsymbol{x}_{U_2}^{(k)} = \boldsymbol{x}_{U_1 \cap U_2}^{(k)}$ and $\bar{\boldsymbol{x}}^{(k)}_{U_1} \cdot \bar{\boldsymbol{x}}_{U_2}^{(k)} = \bar{\boldsymbol{x}}_{U_1 + U_2}^{(k)}$ for all $U_1, U_2\in\mathcal{L}(V)$.

Furthermore, we define the inner product $\langle f,g\rangle = \sum_{K\in \qbinomS{V}{k}} f(K) g(K)$ on the set of functions $\qbinomS{V}{k} \to \R$.
For subsets $\mathcal{F}$ and $\mathcal{G}$ of $\qbinomS{V}{k}$ we have $\langle \chi_{\mathcal{F}}, \chi_{\mathcal{G}} \rangle = \#(\mathcal{F} \cap \mathcal{G})$.
Moreover, we use the abbreviation
\[
	\#f = \langle f, \vek{1}\rangle = \sum_{K\in\qbinomS{V}{k}} f(K)\text{,}
\]
noting that
\[
	\#\chi_{\mathcal{F}} = \#\mathcal{F}\text{.}
\]

Let $x,y\in\{0,\ldots,n\}$ with $x\leq y$.
The matrix $W^{(n;x,y)} = W^{(xy)}\in \R^{\qbinomS{V}{x} \times \qbinomS{V}{y}}$ with the entries
\[
	W^{(xy)}_{XY}
	= \begin{cases}
		1 & \text{if }X \leq Y\text{;}\\
		0 & \text{otherwise.}
	\end{cases}
\]
is known as ($x$-vs.-$y$-)\emph{incidence} or \emph{inclusion} matrix.
The following fact has been shown in \cite{Gottlieb-1966-PAMS17[6]:1233-1237} for $q=1$ and \cite{Kantor-1972-MathZeit124[4]:315-318} for $q \geq 2$.
\begin{fact}\label{fact:full_rank}
	For $x \leq y \leq n-x$, the matrix $W^{(xy)}$ has full rank $\qbinomS{n}{x}$.
\end{fact}

The matrix $W^{(xy)}$ and its transpose $(W^{(xy)})^\top$ describe linear maps
\[
	\phi_\downarrow^{(xy)} : \R^{\qbinomS{V}{y}} \to \R^{\qbinomS{V}{x}}
	\quad\text{and}\quad
	\phi_\uparrow^{(yx)} : \R^{\qbinomS{V}{x}} \to \R^{\qbinomS{V}{y}}\text{.}
\]
By Fact~\ref{fact:full_rank}, the maps $\phi_\downarrow^{(xy)}$ are surjective, and the maps $\phi_\uparrow^{(yx)}$ are injective.

A straightforward double counting argument (see \cite[Lem.~3.3]{Kiermaier-Wassermann-2023-GlasMatSerIII58[2]:181-200} in a more general context) shows
\begin{lemma}\label{lem:phi_compose}
Let $0 \leq x \leq y \leq z \leq n$ be integers.
Then
\[
	\phi_\downarrow^{(xy)}\circ\phi_\downarrow^{(yz)} = \qbinomS{z-x}{y-x}\, \phi_\downarrow^{(xz)}
	\quad\text{and}\quad
	\phi_\uparrow^{(zy)}\circ\phi_\uparrow^{(yx)} = \qbinomS{z-x}{y-x}\, \phi_\uparrow^{(zx)}\text{.}
\]
\end{lemma}

\subsection{Designs}
A set $\mathcal{D} \subseteq \qbinomS{V}{k}$ is called a $t$-$(n,k,\lambda)_q$ design if each $T\in\qbinomS{V}{t}$ is contained in exactly $\lambda$ elements of $\mathcal{D}$.
The number $t$ is called the \emph{strength}, and the elements of $\mathcal{D}$ are called \emph{blocks} of $\mathcal{D}$.
In the case $q=1$, $\mathcal{D}$ is known as a \emph{simple combinatorial design}, and in the $q$-analog situation $q \geq 2$ it is known as a \emph{simple subspace design}.
For the theory of \emph{combinatorial designs} ($q = 1$) the reader is referred to \cite{Beth-Jungnickel-Lenz-1999-DesignTheoryI,Beth-Jungnickel-Lenz-1999-DesignTheoryII} and for the $q$-analog situation of the \emph{subspace designs} ($q \geq 2$) to \cite{Braun-Kiermaier-Wassermann-2018-SignalsCommunTechnol:171-211}.

Spelling out the design property for the characteristic function $\chi_{\mathcal{D}}$ of $\mathcal{D}$, we get the following generalization.
For $\lambda\in\R$, a function $f : \qbinomS{V}{k} \to \R$ is called a \emph{real $t$-$(n,k,\lambda)_q$ (subspace) design} if $\langle \boldsymbol{x}^{(k)}_T, f\rangle = \lambda$ for each $T\in\qbinomS{V}{t}$.
Via characteristic functions, the simple designs defined above correspond to the real designs $f$ which are Boolean functions.
For $\im f \subseteq \Z$, the function $f$ is called a \emph{signed (subspace) design}, and for $\im f \subseteq \N$, it is called an (unqualified) \emph{(subspace) design}.
To avoid ambiguities, we may call the latter also \emph{possibly non-simple (subspace) design}.
A design with $\lambda = 0$ is called \emph{null design} or \emph{trade}.

Clearly, the set of real $t$-$(v,k,0)_q$ null designs equals $\ker\phi_{\downarrow}^{(tk)}$, and in general the real $t$-$(v,k,\lambda)_q$ designs are the solutions $f$ of the linear equation system $\phi_{\downarrow}^{(tk)}(f) = \lambda \vek{1}$.

\begin{fact}\label{fact:reduced_design}
	Let $f : \qbinomS{V}{k} \to \R$ be a real $t$-$(n,k,\lambda)_q$ design.
	Then for each $s\in\{0,\ldots,t\}$, $f$ is a real $s$-$(n,k,\lambda_s)_q$ design with
	\[
		\lambda_s = \frac{\qbinomS{n-s}{t-s}}{\qbinomS{k-s}{t-s}} \cdot \lambda = \frac{\qbinomS{n-s}{k-s}}{\qbinomS{n-t}{k-t}}\cdot \lambda\text{.}
	\]
	In particular, $\#f = \lambda_0$.
\end{fact}

For each $t$, $n$ and $k$ there exist \emph{trivial simple designs}, which are the \emph{empty design} $\mathcal{D} = \emptyset$ with parameters $t$-$(n,k,0)$ and the \emph{complete design} $\mathcal{D} = \qbinom{V}{k}{q}$ with parameters $t$-$(n,k,\lambda_{\max})$ where $\lambda_{\max} = \qbinom{n-t}{k-t}{q}$.

An important question in design theory is to decide if a given parameter set $t$-$(n,k,\lambda)_q$ is \emph{realizable} as a simple/unqualified/signed design, meaning that a design with these parameters does exist.
By the above fact, a necessary condition for realizability as a signed design is the \emph{integrality} (or \emph{divisibility}) \emph{condition}, which states that all the associated numbers $\lambda_0,\ldots,\lambda_t$ are integers.
In that case the parameters $t$-$(n,k,\lambda)_q$ are called \emph{admissible}.
It is worth noting that not all admissible parameters are realizable as a simple design. For $q=1$, the smallest (in terms of $n$) admissible, non-realizable parameter set is $3$-$(11,5,2)_1$, see \cite{Dehon-1976-DM15[1]:23-25} and for a simplified argument \cite[Thm.~1]{Kiermaier-Pavcevic-2015-JCD23[11]:463-480}.
In contrast to that, for signed designs, the question of realizability has been settled completely.

\begin{fact}\label{fact:signed_designs_realizable}
	There exists a signed $t$-$(v,k,\lambda)_q$ design if and only if the parameters are admissible.
\end{fact}

\begin{proof}
	See \cite{Graver-Jurkat-1973-JCTSA15[1]:75-90, Wilson-1973-UtilitasM4:207-215} for $q=1$ and \cite{RayChaudhuri-Singhi-1989-LAA114_115:57-68} for $q \geq 2$.
\end{proof}

For fixed $t$, $n$ and $k$, denote the smallest positive integer $\lambda$ such that $t$-$(n,k,\lambda)_q$ is admissible by $\lambda_{\min}$.
The admissible parameter sets $t$-$(n,k,\lambda)_q$ are precisely those with $\lambda_{\min} \mid \lambda$.

To our surprise, we are not aware about any further investigation of the number $\lambda_{\min}$ in the literature.
So to the best of our knowledge, the following result is new.

\begin{lemma}\label{lem:lambda_min}
We have
\begin{equation}
	\lambda_{\min}
	= \frac{\qbinomS{n-t}{k-t}}{\gcd\left(\qbinomS{n-0}{k-0}, \qbinomS{n-1}{k-1}, \ldots, \qbinomS{n-t}{k-t}\right)}
	= \frac{\lambda_{\max}}{\gcd\left(\qbinomS{n-0}{k-0}, \qbinomS{n-1}{k-1}, \ldots, \qbinomS{n-t}{k-t}\right)}\text{.}
	\label{eq:lambda_min_max}
\end{equation}
As polynomials in $\Z[q]$, the factorization of the gcd-part into irreducibles is
\begin{equation}\label{eq:lambda_min}
	{\textstyle
	\gcd\left(\qbinomS{n-0}{k-0}, \qbinomS{n-1}{k-1}, \ldots, \qbinomS{n-t}{k-t}\right)}
	= \prod_{\substack{i\in\Z_{>0} \\ t\leq (n \bmod i) < (k\bmod i)}} \Phi_i\text{,}
\end{equation}
where $\Phi_i\in\Z[q]$ denotes the $i$-th cyclotomic polynomial.
\end{lemma}

\begin{proof}
	A closer look at Fact~\ref{fact:reduced_design} gives the first expression.
	The second equality is a consequence of the $q$-analog of Kummer's theorem (Equation~\eqref{eq:q_kummer}), which for each $s\in\{0,\ldots, t\}$ yields the $\Z[q]$-factorization
	\[
		\qbinomS{n-s}{k-s}=  \prod_{\substack{i\in\Z_{>0} \\ ((n-s) \bmod i) < ((k-s)\bmod i)}} \Phi_i\text{.}
	\]
	Now fix a value $i\in\Z_{>0}$.
	Integral division yields quotients $a_n, a_k\in\Z_{>0}$ and remainders $r_n, r_k\in\{0,\ldots,i-1\}$ such that $n = a_n i + r_n$ and $k = a_k i + r_k$.
	Note that $r_n = (n \bmod i)$ and $r_k = (k \bmod i)$.
	We claim that the following two statements are equivalent.
	\begin{enumerate}[(i)]
		\item\label{AAA} $((n-s) \bmod i) < ((k-s)\bmod i)$ for all $s\in\{0,\ldots,t\}$;
		\item\label{BBB} $t\leq (n \bmod i) < (k\bmod i)$.
	\end{enumerate}

	First assume that~\ref{BBB} is true and fix some $s\in\{0,\ldots,t\}$.
	Then $n-s = a_n i + (r_n - s)$ and $k-s = a_k i + (r_k - s)$ and we have $0 \leq r_n - t \leq r_n - s < r_k - s \leq i-1$.
	Hence $((n-s) \bmod i) = r_n - s < r_k - s = ((k-s) \bmod i)$.

	Now assume that~\ref{BBB} is false.
	If $(n \bmod i) \geq (k\bmod i)$, then $\ref{AAA}$ fails for $s = 0$.
	It remains to consider the case $(n \bmod i) < (k\bmod i)$ and $t > r_n$.
	We show that \ref{AAA} fails for $s = r_n+1 \leq t$.
	Here we have $n-s = a_n i + r_n - s = a_n i - 1 = (a_n - 1)i + (i-1)$ and thus $((n - s) \bmod i) = i-1 \geq ((k - s)\bmod i)$.
\end{proof}

\section{Algebraic and geometric properties}\label{sec:alg_geom}
Already in the original article, Cameron-Liebler sets $\mathcal{L} \subseteq \qbinomS{\F_q^4}{2}$ of lines have been characterized by several equivalent properties \cite[Prop.~3]{Cameron-Liebler-1982-LAA46:91-102}.%
\footnote{Extended and generalized characterizations are found in several other articles, like~\cite[Th.~1]{Penttila-1991-GeomDed37[3]:245-252}, \cite[Th.~3.7]{Rodgers-Storme-Vansweevelt-2018-Comb38[3]:739-757}, \cite[Th.~2.9]{Blokhuis-DeBoeck-DHaeseleer-2019-DCC87[8]:1839-1856}, and \cite[Prop.~10]{DeBeule-DHaeseleer-Ihringer-Mannaert-2023-ElecJComb30[1]:P1.31}.}
Property~(i) says that $\chi_{\mathcal{L}}$ is contained in $\im\phi_{\uparrow}^{(21)}$, and properties~(iv) and~(v) state that $\mathcal{L}$ has constant intersection with any (regular) line spread of $\PG(\F_q^4)$.
Intuitively, the former property is of algebraic and the latter ones are of geometric nature.
In the following, we want to find suitable generalizations of these properties.

In this section, let $V$ be a set or a $\F_q$-vector space of finite rank $n = \rk V$.
Moreover, we fix integers $k\in\{0,\ldots,n\}$ and $t\in \{0,\ldots,\min(k,n-k)\}$.
We define the space
\[
	\bar{V}^{(k)}_t \coloneqq \bar{V}_t \coloneqq \rowsp(W^{(tk)}) = \im \phi_{\uparrow}^{(kt)} \subseteq \R^{\qbinomS{V}{k}}\text{.}
\]
By Fact~\ref{fact:full_rank}, we have
\[
	\dim\bar{V}_t = \qbinomS{n}{t}\text{.}
\]
The natural generalization of the algebraic property is $f\in \bar{V}_t$.

\begin{remark}
In the language of association schemes, for $k \leq \frac{n}{2}$, the above definition precisely means that $f$ is a $t$-\emph{antidesign} in the ($q$-)Johnson scheme $J_q(n,k)$~\cite{Roos-1982-DelftProgrRep7[2]:98-109}.
As we will see, in this sense the section title \emph{algebraic and geometric properties} could really also be regarded as \emph{antidesigns and designs}.
\end{remark}

Searching for suitable generalizations of the geometric properties, we note that a line spread in $\PG(\F_q^4)$ is the same as a simple $1$-$(4,2,1)_q$ design.
For $\lambda\in\R$, we will denote the set of all real $t$-$(n,k,\lambda)_q$ designs by $U_{t,\lambda}$.
The $\lambda$-value of the complete $t$-design $\qbinomS{V}{k}$ is $\lambda_{\max} = \qbinomS{n-t}{k-t}$.

For the set of all real $t$-$(n,k,0)_q$ null designs we have $U_{t,0} = \ker W^{(tk)}$, so $U_{t,0} = \bar{V}_t^\perp$ and $\dim U_{t,0} = \qbinomS{n}{k} - \qbinomS{n}{t}$.
For fixed $\lambda\in \R$, the scaled complete design $\frac{\lambda}{\lambda_{\max}}\cdot \boldsymbol{1}$ clearly is a real $t$-$(n,k,\lambda)_q$ design.
Hence, as the space of all solutions $f$ of the linear equation system $W^{(tk)} f = \lambda\vek{1}$, the set of all real $t$-$(n,k,\lambda)_q$ designs is the affine subspace
\[
	U_{t,\lambda} = \frac{\lambda}{\lambda_{\max}}\cdot \vek{1} + U_{t,0}\text{.}
\]
So, $\delta$ is a real $t$-$(n,k,\lambda)_q$ design if and only if $\delta - \frac{\lambda}{\lambda_{\max}}\cdot \vek{1} \in U_{t,0}$.
By $U_{t,0} = \bar{V}_t^\perp$, this implies
\begin{equation}
	U_{t,\lambda} = \big\{\delta : \qbinomS{V}{k} \to \R \mid \langle f, \delta\rangle = \frac{\lambda}{\lambda_{\max}}\cdot\#f \text{ for all }f\in\bar{V}_t\big\}\text{.} \label{eq:designs_via_vbar}
\end{equation}
The space $U_{t,{\ast}}$ of all $t$-$(n,k,\lambda)_q$ designs with any value of $\lambda\in\R$ is $U_{t,{\ast}} \coloneqq U_{t,0} + \langle \vek{1}\rangle_{\R}$ of dimension $\dim U_{t,{\ast}} = \qbinomS{n}{k} - \qbinomS{n}{t} + 1$.
Its orthogonal space is $U_{t,{\ast}}^\perp = U_{t,0}^\perp \cap \langle\vek{1}\rangle_{\R}^\perp = \bar{V}_t \cap \langle\vek{1}\rangle_{\R}^\perp$, which equals the set of all $f\in\bar{V}_t$ with $\#f = 0$.
We will denote the $\lambda$-value of $\delta\in U_{t,{\ast}}$ by $\lambda_{\delta}$.

\begin{lemma}\label{lem:alg_geom_core}
	Let $\Delta\subseteq U_{t,{\ast}}$ and let
	\[
		F
		= \Big\{f : \qbinomS{V}{k} \to \R \mid \langle f, \delta\rangle = \frac{\lambda_{\delta}}{\lambda_{\max}} \cdot \# f \text{ for all }\delta\in \Delta\Big\}\text{.}
	\]
	Then $\bar{V}_t \subseteq F$ and $\dim F + \dim \langle \Delta \cup \{\vek{1}\}\rangle_{\R} = \qbinomS{n}{k} + 1$.
\end{lemma}

\begin{proof}
	The inclusion $\bar{V}_t \subseteq F$ follows from Equation~\ref{eq:designs_via_vbar}.

	We have $\langle f,\delta\rangle = \frac{\lambda_{\delta}}{\lambda_{\max}} \cdot \# f$ if and only if $\langle f, \delta - \frac{\lambda_{\delta}}{\lambda_{\max}}\cdot \vek{1}\rangle = 0$.
	Therefore, $F = (\Delta')^\perp$ where $\Delta' \coloneqq \{\delta - \frac{\lambda_{\delta}}{\lambda_{\max}}\cdot \vek{1} \mid \delta\in \Delta\}$.
	Hence $\dim F + \dim \langle\Delta'\rangle_{\R} = \dim\R^{\qbinomS{V}{k}} = \qbinomS{n}{k}$.
	Since all elements of $\Delta'$ are null designs, $\vek{1} \notin\langle\Delta'\rangle_{\R}$ and therefore $\dim\langle\Delta'\rangle_{\R} = \dim \langle\Delta' \cup\{\vek{1}\}\rangle_{\R} - 1$.
	The observation $\langle\Delta'\cup\{\vek{1}\}\rangle_{\R} = \langle\Delta\cup\{\vek{1}\}\rangle_{\R}$ concludes the proof.
\end{proof}

The above considerations are loosely known by the term \emph{design orthogonality}, attributed to Delsarte \cite{Delsarte-1977-PhilRR32[5-6]:373-411}, see also~\cite{Roos-1982-DelftProgrRep7[2]:98-109} for an accessible approach.
The following statement describes the essential correspondence between algebraic and geometric properties.

\begin{corollary}\label{cor:alg_geom_essence}
	Let $\Delta$ be a set of real $t$-$(n,k,{\ast})_q$ designs.
	Then the following are equivalent:
	\begin{enumerate}[(i)]
		\item\label{cor:alg_geom_essence:alg} $\bar{V}_t = \big\{f : \qbinomS{V}{k} \to \R \mid \langle f, \delta \rangle = \frac{\lambda_{\delta}}{\lambda_{\max}}\cdot \#f\text{ for all }\delta\in\Delta\big\}$.
		\item\label{cor:alg_geom_essence:geom} $\langle \Delta \cup \{\vek{1}\} \rangle_{\R} = U_{t,{\ast}}$.
	\end{enumerate}
	In particular,
\[
	\bar{V}_t = \big\{f : \qbinomS{V}{k} \to \R \mid \langle f, \delta\rangle = \frac{\lambda_\delta}{\lambda_{\max}} \cdot \#f \text{ for all }\delta\in U_{t,{\ast}}\big\}\text{.}\label{eq:deg_intersection}
\]
\end{corollary}

\begin{proof}
	This is a direct consequence of Lemma~\ref{lem:alg_geom_core} combined with $\dim\bar{V}_t = \qbinomS{n}{t}$ and $\dim U_{t,{\ast}} = \qbinomS{n}{k} - \qbinomS{n}{t} + 1$.
\end{proof}

To apply Corollary~\ref{cor:alg_geom_essence}, we need suitable sets $\Delta$.

\begin{lemma}\label{lem:delta}
	Let $\Delta$ be one of the following:
	\begin{enumerate}[(a)]
		\item\label{lem:delta:0} The set of all signed $t$-$(n,k,0)_q$ null designs or
		\item\label{lem:delta:nonsimple} the set of all possibly non-simple $t$-$(n,k,\lambda)_q$ designs with $\lambda\in\N$.
	\end{enumerate}
	Then $U_{t,{\ast}} = \langle\Delta \cup\{\vek{1}\}\rangle$.
\end{lemma}

\begin{proof}
Part~\ref{lem:delta:0}:
Since the entries of $W^{(tk)}$ are rational, the space $U_{t,0} = \ker W^{(tk)}$ has a rational basis.
Multiplication with common denominators yields an integral basis.
Hence $\Delta$ contains a basis of $U_{t,0}$ and thus $\langle \Delta\cup\{\vek{1}\}\rangle = U_{t,{\ast}}$.

Part~\ref{lem:delta:nonsimple}:
By the above argument, there exists an integral basis $B$ of $U_{t,0}$.
Now replace every element $\vek{v}\in B$ by $\vek{v}' = \vek{v} + a\vek{1}$ with a suitable integer $a$ such that all the entries of $\vek{v}'$ are non-negative.
The resulting set $B'$ of vectors is contained in $\Delta$ and has the property $\langle B' \cup\{\vek{1}\} \rangle = U_{t,{\ast}}$.
\end{proof}

\pagebreak

\begin{theorem}\label{thm:alg_geom}
	Let $f : \qbinomS{V}{k} \to \R$.
	The following are equivalent.\\[0.5\baselineskip]
	\noindent\emph{Algebraic property:}
	\begin{enumerate}[(i)]
		\item\label{thm:alg_geom:alg} $f \in \bar{V}_t$.
	\end{enumerate}
	\emph{Geometric properties:}
	\begin{enumerate}[resume*]
		\item\label{thm:alg_geom:geom_full} There is a constant $c\in\R$ such that $\langle f, \delta\rangle = \lambda c$ for all real $t$-$(n,k,\lambda)_q$ designs with $\lambda\in\R$.
		\item\label{thm:alg_geom:geom_null} $\langle f, \delta\rangle = 0$ for all signed $t$-$(n,k,0)_q$ null designs $\delta : \qbinomS{V}{k} \to \Z$.
		\item\label{thm:alg_geom:geom_nonsimple} There is a constant $c\in\R$ such that $\langle f, \delta\rangle = \lambda c$ for all possibly non-simple $t$-$(n,k,\lambda)_q$ designs $\delta : \qbinomS{V}{k} \to \N$ with $\lambda\in\N$.
	\end{enumerate}
	The constant in properties~\ref{thm:alg_geom:geom_full} and~\ref{thm:alg_geom:geom_nonsimple} necessarily equals $c = \frac{1}{\lambda_{\max}}\cdot \#f$.
\end{theorem}

\begin{proof}
Applying properties~\ref{thm:alg_geom:geom_full} and~\ref{thm:alg_geom:geom_nonsimple} to the complete $t$-$(n,k,\lambda_{\max})_q$ design (which is a simple design), we see that the constant equals $c = \frac{1}{\lambda_{\max}}\cdot \#f$.
The equivalence of property~\ref{thm:alg_geom:alg} and the three geometric properties follows from Corollary~\ref{cor:alg_geom_essence} and Lemma~\ref{lem:delta}.
\end{proof}

\begin{remark}
	It feels tempting to restrict the geometric property \ref{thm:alg_geom:geom_nonsimple} to the set $\Delta$ of all \emph{simple} designs $\delta : \qbinomS{V}{k} \to \{0,1\}$.
	By Corollary~\ref{cor:alg_geom_essence}, this is possible exactly for those parameter tuples $(q,n,k,t)$ such that the \emph{richness condition} $\langle\Delta\rangle_{\R} = U_{t,{\ast}}$ is met (note that $\vek{1}\in\Delta$).

	Unfortunately, this condition does not hold in all cases.
	For a counterexample, consider $q=1$, $n = 10$, $k = 5$ and $t=4$.
	Here $\lambda_{\min} = \lambda_{\max} = 6$, so the only simple $4$-$(10,5,\lambda)_1$ designs are the empty ($\lambda = 0$) and the complete ($\lambda = 6$) design.
	This implies $\dim \langle\Delta \cup \{\vek{1}\}\rangle_{\R} = 1$, which is too small.
\end{remark}

\section{Weights and degree}\label{sec:weights_degree}
In this section, let $V$ be either a set or a finite $\F_q$-vector space of finite rank.
The rank of $V$ will be denoted by $n = \rk V$.

Let $i\in\{0,\ldots,\min(k,n-k)\}$.
(Equivalently, $i$ is a non-negative integer with $i \leq k \leq n-i$.)
The set of $i$-pencils $\mathcal{X}_i^{(k)} = \{\boldsymbol{x}_I^{(k)} \mid I\in\qbinomS{V}{i}\}$ spans the row space
\[
	\bar{V}^{(k)}_i = \rowsp(W^{(ik)}) = \im \phi_{\uparrow}^{(ki)} \subseteq \R^{\qbinomS{V}{k}}\text{.}
\]
By Fact~\ref{fact:full_rank}, $\dim\bar{V}^{(k)}_i = \qbinomS{n}{i}$ and $\mathcal{X}_i^{(k)}$ is a basis of $\bar{V}^{(k)}_i$.
Therefore, each function $f\in \bar{V}^{(k)}_i$ is a unique linear combination $f = \sum_{I\in\qbinomS{V}{i}} a_I \boldsymbol{x}^{(k)}_I$ of $i$-pencils.
In this case, the coefficients $a_I\in\R$ will be called the \emph{weights} or the \emph{$i$-weights} of $f$ (or of $\mathcal{F}$ in the case $f = \chi_{\mathcal{F}}$), and $\wt^{(i)}_f : \qbinomS{V}{i} \to \R$, $I \mapsto a_I$ is the \emph{$i$-weight distribution} $\wt^{(i)}_f$ of $f$ (or of $\mathcal{F}$), which may also be denoted as the sequence $(a_I)_{I\in\qbinomS{V}{i}}$ .
It holds that $\phi_\uparrow^{(ki)}(\wt^{(i)}_f) = f$.
Note that for $\Q$-valued functions $f\in\bar{V}^{(k)}_i$, also the $i$-weights are in $\Q$.

\begin{lemma}\label{lem:vi_chain}
Let $k\in\{0,\ldots,n\}$.
Then
\[
	\langle\boldsymbol{1}\rangle = \bar{V}^{(k)}_0 \subseteq \bar{V}^{(k)}_1 \subseteq \ldots \subseteq \bar{V}^{(k)}_{\min(k,n-k)} = \R^{\qbinomS{V}{k}}\text{,}
\]
where $\boldsymbol{1}$ denotes the all-one function $K\mapsto 1$.

For integers $0\leq i\leq j \leq \min(k,n-k)$ and $f\in \bar{V}^{(k)}_i$, the $j$-weights of $f$ are given by
\[
	\wt^{(j)}_f = \qbinomS{k-i}{j-i}^{-1} \phi_\uparrow^{(ji)}(\wt_f^{(i)})\text{.}
\]
\end{lemma}

\begin{proof}
The equality $\bar{V}^{(k)}_0 = \langle\boldsymbol{1}\rangle$ follows from $\boldsymbol{x}_{\boldsymbol{0}}^{(k)} = \boldsymbol{1}$, and the equality $\bar{V}^{(k)}_{\min(k,n-k)} = \R^{\qbinomS{V}{k}}$ from $\dim \bar{V}^{(k)}_{\min(k,n-k)} = \qbinomS{n}{k} = \dim \R^{\qbinomS{V}{k}}$.

For integers $0\leq i\leq j \leq \min(k,n-k)$ and $f\in \bar{V}^{(k)}_i$, we have
\[
	f
	= \phi_\uparrow^{(ki)}(\wt^{(i)}_f)
	= \qbinomS{k-i}{j-i}^{-1}\,\phi_\uparrow^{(kj)}(\phi_\uparrow^{(ji)}(\wt_f^{(i)}))\text{.}
\]
\end{proof}

\begin{remark}\label{rem:vi_vs_johnson}
	Equipping the ambient space with the standard inner product, each $\bar{V}_{t-1}$ has a unique orthogonal complement $V_{t}$ in $\bar{V}_t$ (where $t \in \{0,\ldots,\min(k,n-k)\}$, assuming that $\bar{V}_{-1} = \{\vek{0}\}$).
	In this way, the chain
	\[
		\langle\boldsymbol{1}\rangle = \bar{V}^{(k)}_0 \subseteq \bar{V}^{(k)}_1 \subseteq \ldots \subseteq \bar{V}^{(k)}_{\min(k,n-k)} = \R^{\qbinomS{V}{k}}
	\]
	of row spaces of Lemma~\ref{lem:vi_chain} yields the orthogonal decomposition
	\[
		\bar{V}_t = V_0 \perp \ldots \perp V_t\text{.}
	\]
	and in particular
	\[
		\R^{\qbinomS{V}{k}} = V_0 \perp V_1 \perp \ldots \perp V_{\min(k,n-k)}\text{.}
	\]
	In the case $k \leq \frac{n}{2}$, this is exactly the decomposition of the Johnson (for $q=1$) or $q$-Johnson (for $q \geq 2$) scheme $J_q(n,k)$ into (suitably ordered) orthogonal eigenspaces.
	The $q$-Johnson scheme is a $k$-class association scheme on the base set $\qbinomS{V}{k}$.
	For a deeper introduction to the theory of association schemes, the reader is referred to the textbooks \cite{Bannai-Bannai-Ito-Tanaka-2021,Brouwer-Cohen-Neumaier-1989}.
\end{remark}

Lemma~\ref{lem:vi_chain} motivates the definition of the \emph{degree} $\deg(f) = \deg_V(f)$ of a function $f : \qbinomS{V}{k} \to \R$, $f\neq\vek{0}$, as the smallest number $t\in\{0,\ldots,\min(k,n-k)\}$ such that $f\in \bar{V}^{(k)}_t$.%
\footnote{In existing definitions in the literature like in \cite{DeBeule-DHaeseleer-Ihringer-Mannaert-2023-ElecJComb30[1]:P1.31}, the property of $t$ being the \emph{smallest} possible integer may be missing.
We decided to use this refined definition since we feel that the degree should be a unique number.
In this way, we get an analogous setting to the degree of a polynomial.}
Like for polynomials, the degree of the zero function is defined as $\deg(\vek{0}) = -\infty$.
In other words, $\bar{V}^{(k)}_t$ is the set of all functions $\qbinomS{V}{k} \to \R$ of degree at most $t$.
By Lemma~\ref{lem:vi_chain}, a function $f$ is of degree $0$ if and only if $f$ is constant and non-zero.

\begin{lemma}\label{lem:deg_operations}
	Let $k\in\{0,\ldots,n\}$.
	Let $f,g : \qbinomS{V}{k} \to \R$ and $\lambda\in\R$.
	Then
	\begin{enumerate}[(a)]
		\item\label{lem:deg_operations:scalar_mul} $\deg(\lambda f) \leq \deg(f)$ with equality if and only if $\lambda \neq 0$ or $f = \vek{0}$.
		\item\label{lem:deg_operations:sum} $\deg(f \pm g) \leq \max(\deg(f), \deg(g))$. Whenever $\deg(f) \neq \deg(g)$, we have equality.
		\item\label{lem:deg_operations:prod} $\deg(fg) \leq \deg(f) + \deg(g)$.
	\end{enumerate}
\end{lemma}

\begin{proof}
	The first two statements follow from the fact that $\bar{V}^{(k)}_i$ is an $\R$-vector space, and from Lemma~\ref{lem:vi_chain}.
	For the equality statement in~\ref{lem:deg_operations:sum}, without loss of generality $\deg(g) \leq \deg(f)$ and hence $\deg(f\pm g) \leq \max(\deg(f), \deg(g)) = \deg(f)$.
	From $\deg(f) = \deg((f\pm g) \mp g) \leq \max(\deg(f\pm g), \deg(g))$ we see that we can't have $\deg(f \pm g) < \deg(f)$ and $\deg(g) < \deg(f)$ at the same time.

	The proof of the third statement will be given later at the end of Section~\ref{sec:weights_degree:ambient}, when we have the required basics.
\end{proof}

We remark that in Lemma~\ref{lem:deg_operations}, the strict inequalities $\deg(f + g) < \min(\deg(f),\deg(g))$ and $\deg (fg) < \deg(f) + \deg(g)$ may occur.
For the first inequality, consider $\chi_{\mathcal{F}} + \chi_{\mathcal{F}^\complement} = \boldsymbol{1}$ and for the second inequality, consider the idempotent functions $\chi_{\mathcal{F}} \cdot \chi_{\mathcal{F}} = \chi_{\mathcal{F}}$.

Now let $\mathcal{F}\subseteq \qbinomS{V}{k}$ be a set of $k$-subspaces.
The elements of $\mathcal{F}$ will be called \emph{blocks}.
For simplicity, we may identify a set $\mathcal{F}$ with its characteristic Boolean function $\chi_{\mathcal{F}}$.
In this way, we get the \emph{degree} $\deg\mathcal{F} = \deg_{\chi_\mathcal{F}}\in\{0,\ldots,\min(k,n-k)\}$ and (for all $i\in\{0,\ldots,n\}$ it exists) the \emph{weight distribution} $\wt^{(i)}_\mathcal{F} = \wt^{(i)}_{\chi_\mathcal{F}}$ of the set $\mathcal{F}$.
We have $\deg(\emptyset) = -\infty$.
Clearly, $\qbinomS{V}{k}$ is the only set of degree $0$.
We define the \emph{supplementary} set of $\mathcal{F}$ as $\mathcal{F}^\complement = \qbinomS{V}{k} \setminus \mathcal{F}$.%
\footnote{We prefer \enquote{supplement} over \enquote{complement}, since this is the established terminology in design theory.}

\begin{lemma}\label{lem:set_op}
	Let $\mathcal{F},\mathcal{F}_1, \mathcal{F}_2\subseteq\qbinomS{V}{k}$.
	\begin{enumerate}[(a)]
		\item\label{lem:set_op:meet_join}
		$\deg(\mathcal{F}_1 \cap \mathcal{F}_2) \leq \deg\mathcal{F}_1 + \deg\mathcal{F}_2$ and $\deg(\mathcal{F}_1 \cup \mathcal{F}_2) \leq \deg\mathcal{F}_1 + \deg\mathcal{F}_2$.
		\item\label{lem:set_op:disjoint_join}
		If $\mathcal{F}_1 \cap \mathcal{F}_2 = \emptyset$, then $\deg(\mathcal{F}_1 \cup \mathcal{F}_2) \leq \max(\deg\mathcal{F}_1,\deg\mathcal{F}_2)$.\\
		Whenever $\deg\mathcal{F}_1 \neq \deg\mathcal{F}_2$, we have equality.
		\item\label{lem:set_op:diff}
		If $\mathcal{F}_1 \subseteq \mathcal{F}_2$, then $\deg(\mathcal{F}_2 \setminus \mathcal{F}_1) \leq \max(\deg\mathcal{F}_1,\deg\mathcal{F}_2)$.\\
		Whenever $\deg\mathcal{F}_1 \neq \deg\mathcal{F}_2$, we have equality.
		\item\label{lem:set_op:supplement}
		If $\mathcal{F}$ is not one of the border cases $\emptyset$ and $\qbinomS{V}{k}$, then $\deg(\mathcal{F}^\complement) = \deg(\mathcal{F})$.
	\end{enumerate}
\end{lemma}

\begin{proof}
	Part~\ref{lem:set_op:meet_join} follows from Lemma~\ref{lem:deg_operations}\ref{lem:deg_operations:sum} and~\ref{lem:deg_operations:prod} by $\chi_{\mathcal{F}_1 \cap \mathcal{F}_2} = \chi_{\mathcal{F}_1} \cdot \chi_{\mathcal{F}_2}$ and $\chi_{\mathcal{F}_1 \cup \mathcal{F}_2} = \chi_{\mathcal{F}_1} + \chi_{\mathcal{F}_2} - \chi_{\mathcal{F}_1} \cdot \chi_{\mathcal{F}_2}$.

	Part~\ref{lem:set_op:disjoint_join} follows from Lemma~\ref{lem:deg_operations}\ref{lem:deg_operations:sum} by $\chi_{\mathcal{F}_1 \cup \mathcal{F}_2} = \chi_{\mathcal{F}_1} + \chi_{\mathcal{F}_2}$ for disjoint sets $\mathcal{F}_1$ and $\mathcal{F}_2$.

	Part~\ref{lem:set_op:diff} follows from Lemma~\ref{lem:deg_operations}\ref{lem:deg_operations:sum} by $\chi_{\mathcal{F}_2 \setminus \mathcal{F}_1} = \chi_{\mathcal{F}_2} - \chi_{\mathcal{F}_1}$.

	For Part~\ref{lem:set_op:supplement}, $\deg(\mathcal{F}^\complement) = \deg(\qbinomS{V}{k} \setminus \mathcal{F}) \leq \max(\deg(\qbinomS{V}{k}), \deg(\mathcal{F})) = \max(0,\deg(\mathcal{F})) = \deg(\mathcal{F})$ and in the same way $\deg(\mathcal{F}) = \deg((\mathcal{F}^\complement)^\complement) \leq \deg(\mathcal{F}^\complement)$.
\end{proof}

An important question is the characterization of the possible sizes $\#\mathcal{F}$ of sets $\mathcal{F}$ of degree $t$, and in particular, the smallest possible size $m_q(n,k,t)$ of a nontrivial set $\mathcal{F}$ of degree $\leq t$.

\begin{remark}\label{rem:m_q}
By dualization (see Theorem~\ref{thm:dual_v_space}), $m_q(n,k,t) = m_q(n,n-k,t)$, therefore we may restrict the investigation to $k \leq \frac{n}{2}$.
By $\deg \vek{x}^{(k)}_t \leq t$ we have $m_q(n,k,t) \leq \#\vek{x}^{(k)}_t = \qbinomS{n-t}{k-t}$.

For degree $1$ and $k \leq \frac{n}{2}$, the bound is always sharp:
For $q = 1$, the Boolean functions of degree~$1$ have been classified in~\cite[Th.~1.2]{Filmus-Ihringer-2019-JCTSA162:241-270}, essentially showing that only point pencils and their duals do exist.
For $q \geq 2$, the situation is much more complicated, as it covers the classical Cameron-Liebler line classes.
The equality $m_q(n,k,1) = \qbinomS{n-1}{k-1}$ follows from~\cite[Lemma 4.1]{Blokhuis-DeBoeck-DHaeseleer-2019-DCC87[8]:1839-1856}.

For further results on the sizes of small Boolean degree $1$ functions for $q \geq 2$ see \cite{DeBeule-Mannaert-Storme-2024-DCC-prelim, ihringer2024classification}.
\end{remark}

In the following, we give results on the size of a (Boolean) degree $t$ function.

\begin{lemma}\label{lem:size_by_deg}
	Let $f\in\R^{\qbinomS{V}{k}}$ and $e\in\{\max(\deg(f),0),\ldots,\min(k,n-k)\}$.
	Then
	\[
		\#f
		= \qbinom{n-e}{k-e}{q} \#\wt^{(e)}\text{.}
	\]
\end{lemma}

\begin{proof}
	For all $x\in\{0,\ldots,n\}$, the matrix $W^{(xn)}$ is a $(\qbinom{n}{x}{q} \times 1)$ all-one column vector.
	Hence for all $f\in\R^{\qbinomS{V}{x}}$ we have $\phi_{\uparrow}^{(n,x)}(f) = \sum_{X\in\qbinomS{V}{x}} f(X)$.
	Therefore by Lemma~\ref{lem:phi_compose}for
	\begin{multline*}
		\#f = \sum_{K\in\qbinomS{V}{k}} f(K)
		= \phi_{\uparrow}^{(n,k)}(f)
		= (\phi_{\uparrow}^{(n,k)}\circ\phi_{\uparrow}^{(k,e)})(\wt_f^{(e)}) \\
		= \qbinom{n-e}{k-e}{q}\phi_{\uparrow}^{(n,e)}(\wt_{f}^{(e)})
		= \qbinom{n-e}{k-e}{q}\sum_{E\in\qbinomS{V}{e}} \wt_{f}^{(e)}(E)
		= \qbinom{n-e}{k-e}{q}\#\wt_f^{(e)}\text{.}
	\end{multline*}
\end{proof}

The connection to the theory of designs yields the following divisibility property.

\begin{theorem}\label{thm:size_divisibility}
	Let $f : \qbinomS{V}{k} \to \Z$ be a non-zero function of degree $t$.
	Then
	\[
		\gcd\left(\textstyle\qbinomS{n-0}{k-0},\qbinomS{n-1}{k-1},\ldots,\qbinomS{n-t}{k-t}\right) \mid\#f\text{.}
	\]
\end{theorem}

\begin{proof}
	By the algebraic property, there exists an $\vek{x} : \qbinomS{V}{t} \to \R$ with $\vek{x}^\top W^{(tk)} = f^\top$.
	The complete design gives
	\[
		W^{(tk)} \vek{1}_{\qbinomS{V}{k}} = \lambda_{\max} \vek{1}_{\qbinomS{V}{t}}\text{.}
	\]
	By Fact~\ref{fact:signed_designs_realizable} there exists a signed $t$-$(n,k,\lambda_{\min})_q$ design $\delta$, so
	\[
		W^{(tk)} \delta = \lambda_{\min} \vek{1}_{\qbinomS{V}{t}}\text{.}
	\]
	Left multiplication of the last two equations by $\vek{x}^\top$ yields
	\[
		\#f = \lambda_{\max}\cdot \#\vek{x}
		\quad\text{and}\quad
		\langle f,\delta\rangle = \lambda_{\min}\cdot \#\vek{x}
	\]
	and hence
	\[
		\#f = \frac{\lambda_{\max}}{\lambda_{\min}}\cdot \langle f,\delta\rangle\text{.}
	\]
	By $\langle f, \delta\rangle\in\Z$ and $\frac{\lambda_{\max}}{\lambda_{\min}} = \gcd\big(\qbinomS{n-0}{k-0},\qbinomS{n-1}{k-1},\ldots,\qbinomS{n-t}{k-t}\big)$ (from Equation~\eqref{eq:lambda_min_max}), the claim follows.
\end{proof}

Note that Theorem~\ref{thm:size_divisibility} reproduces (and in a few cases improves) the divisibility conditions found in~\cite[Appendix~B]{DeBeule-DHaeseleer-Ihringer-Mannaert-2023-ElecJComb30[1]:P1.31}.

Now we focus on the expression $\gcd(\textstyle\qbinomS{n-0}{k-0},\qbinomS{n-1}{k-1},\ldots,\qbinomS{n-t}{k-t})$ in Theorem~\ref{thm:size_divisibility}.

\begin{lemma}\label{lem:gcd_helper}
For $k\geq 1$,
\[
	\gcd\Big(\qbinomS{n}{k},\,\qbinomS{n-1}{k-1}\Big)
	= \frac{\qnumbS{\gcd(n,k)}}{\qnumbS{k}} \cdot\qbinomS{n-1}{k-1}\text{.}
\]
\end{lemma}

\begin{proof}
	\begin{align*}
		\qnumbS{k}\cdot\gcd\Big(\qbinomS{n}{k}, \qbinomS{n-1}{k-1}\Big)
		& = \gcd\Big(\qnumbS{k}\cdot\qbinomS{n}{k},\, \qnumbS{k}\cdot\qbinomS{n-1}{k-1}\Big) \\
		& = \gcd\Big(\qnumbS{n}\cdot \qbinomS{n-1}{k-1},\, \qnumbS{k}\cdot\qbinomS{n-1}{k-1}\Big) \\
		& = \qbinomS{n-1}{k-1}\cdot\gcd(\qnumbS{n},\,\qnumbS{k})
		= \qbinomS{n-1}{k-1}\qnumbS{\gcd(n,k)}\text{.}
	\end{align*}
\end{proof}

\begin{remark}
	Lemma~\ref{lem:gcd_helper} simplifies the $\gcd$-expression in Theorem~\ref{thm:size_divisibility} in the case $t = 1$.
	It is also useful for higher values of $t$, as it can be applied recursively.
	By the same proof as above, for every $a\in\Z$
	\[
		\gcd\Big(\qbinomS{n}{k},\,a\cdot\qbinomS{n-1}{k-1}\Big)
		= \frac{\gcd(\qnumbS{n},a\cdot\qnumbS{k})}{\qnumbS{k}} \cdot\qbinomS{n-1}{k-1}\text{.}
	\]
	Now for $t = 2$ (and hence $k \geq 2$) we get
	\begin{align*}
		\gcd\Big(\qbinomS{n}{k},\,\qbinomS{n-1}{k-1},\,\qbinomS{n-2}{k-2}\Big)
		& = \gcd\Big(\gcd\Big(\qbinomS{n}{k},\,\qbinomS{n-1}{k-1}\Big),\,\qbinomS{n-2}{k-2}\Big) \\
		& = \gcd\Big(\frac{\qnumbS{\gcd(n,k)}}{\qnumbS{k}} \cdot\qbinomS{n-1}{k-1},\,\qbinomS{n-2}{k-2}\Big) \\
		& = \frac{\qnumbS{\gcd(n,k)}}{\qnumbS{k}}\, \gcd\Big(\qbinomS{n-1}{k-1},\,\frac{\qnumbS{k}}{\qnumbS{\gcd(n,k)}}\,\qbinomS{n-2}{k-2}\Big) \\
		& = \frac{\qnumbS{\gcd(n,k)}}{\qnumbS{k}}\cdot\frac{\gcd(\qnumbS{n-1},\frac{\qnumbS{k}\cdot\qnumbS{k-1}}{\qnumbS{\gcd(n,k)}})}{\qnumbS{k-1}}\cdot\qbinomS{n-2}{k-2} \\
		& = \frac{\gcd(\qnumbS{\gcd(n,k)}\cdot\qnumbS{n-1},\, \qnumbS{k}\cdot\qnumbS{k-1})}{\qnumbS{k}\cdot\qnumbS{k-1}}\, \qbinomS{n-2}{k-2}\text{.}
	\end{align*}
\end{remark}

\begin{corollary}\label{cor:m_q_interval}
	Let $0 \leq t \leq k \leq \frac{n}{2}$ and $a = \gcd\left(\textstyle\qbinomS{n-0}{k-0},\qbinomS{n-1}{k-1},\ldots,\qbinomS{n-t}{k-t}\right)$.
	Then
	\[
		m_q(n,k,t) \in \big\{a, 2a, 3a, \ldots, \textstyle\qbinomS{n-t}{k-t}\big\}\text{.}
	\]
\end{corollary}

\begin{proof}
	By Theorem~\ref{thm:size_divisibility} we have $a \mid m_q(n,k,t)$, and by $\deg \vek{x}^{(k)}_t \leq t$ we have $m_q(n,k,t) \leq \#\vek{x}^{(k)}_t = \qbinomS{n-t}{k-t}$.
\end{proof}

Whenever $\gcd\left(\textstyle\qbinomS{n-0}{k-0},\qbinomS{n-1}{k-1},\ldots,\qbinomS{n-t}{k-t}\right) = \qbinomS{n-t}{k-t}$, Corollary~\ref{cor:m_q_interval} yields the exact value of $m_q(n,k,t)$, and also $\#\wt^{(t)}_\mathcal{F}\in\N$ by Lemma~\ref{lem:size_by_deg}.
We can give a sufficient condition for this case.
\begin{lemma}\label{lem:gcd_max}
	Let $t\in \{0,\ldots k\}$ and suppose that $k-i \mid n-i$ for all $i\in\{0,\ldots,t-1\}$.
	Then
	\[
		\gcd\left(\qbinomS{n-0}{k-0},\qbinomS{n-1}{k-1},\ldots,\qbinomS{n-t}{k-t}\right) = \qbinomS{n-t}{k-t}
	\]
	and hence for $k\leq\frac{n}{2}$
	\[
		m_q(n,k,t) = \qbinomS{n-t}{k-t}\text{.}
	\]
\end{lemma}

\begin{proof}
	We apply induction over $t$.
	For $t=0$ the statement is clear.
	For $t \geq 1$ we have
	\begin{align*}
		\gcd\left(\qbinomS{n-0}{k-0},\qbinomS{n-1}{k-1},\ldots,\qbinomS{n-t}{k-t}\right)
		& = \gcd\left(\gcd\Big(\qbinomS{n-0}{k-0},\ldots,\qbinomS{n-t+1}{k-t+1}\Big),\, \qbinomS{n-t}{k-t}\right) \\
		& = \gcd\left(\qbinomS{n-t+1}{k-t+1},\qbinomS{n-t}{k-t}\right) \\
		& = \frac{\qnumbS{\gcd(n-t+1, k-t+1)}}{\qnumbS{k-t+1}} \cdot \qbinomS{n-t}{k-t}
		= \qbinomS{n-t}{k-t}\text{,}
	\end{align*}
	where we used the induction hypothesis in the second step, Lemma~\ref{lem:gcd_helper} in the third step, and $k-t+1 \mid n-t+1$ in the last step.
\end{proof}

\begin{remark}
Note that Lemma~\ref{lem:gcd_max} does not fully classify all cases of parameters with $\gcd\left(\textstyle\qbinomS{n-0}{k-0},\qbinomS{n-1}{k-1},\ldots,\qbinomS{n-t}{k-t}\right) = \qbinomS{n-t}{k-t}$.
For example, this property holds also for $n = 13$, $k = 4$, $t = 2$ and any value of $q$ (and in consequence, we get $m_q(13,4,2)_q = \qbinomS{11}{2}$).

Lemma~\ref{lem:gcd_max} in particular implies that $m_q(n,k,1) = \qbinomS{n-1}{k-1}$ whenever $k \mid n$ and $k\leq\frac{n}{2}$.
This result is not new, see Remark~\ref{rem:m_q}.

As a further example, in all cases with $k \geq 2$, $n = k^2$ and $t = 2$, the condition of Lemma~\ref{lem:gcd_max} are met, such that $m_q(k^2,k,2) = \qbinomS{k^2-2}{k-2}$.

We remark that $m_q(n,k,t)$ with $k\leq\frac{n}{2}$ can indeed be strictly smaller than $\qbinomS{n-t}{k-t}$.
For example, consider the case $q = 1$, $n = 6$, $k = 3$ and $t = 2$.
Corollary~\ref{cor:m_q_interval} yields $m_1(6,3,2) \in \{2,4\}$, and one checks that $\mathcal{F} = \{\{1,2,3\},\{4,5,6\}\}$ is a set of degree $2$.
Hence $m_1(6,3,2) = 2$.
\end{remark}

For Boolean degree $1$ sets $\mathcal{F}$, classifications are usually based on the \emph{parameter} $x\coloneqq\#\mathcal{F}/ \qbinomS{n-1}{k-1}$.
By Lemma~\ref{lem:size_by_deg}, this parameter $x$ equals $\#\wt_{\mathcal{F}}^{(1)}$.
For example, the point-pencils $\boldsymbol{x}^{(k)}_{P}$ with $P\in\qbinomS{V}{1}$ have the parameter $x = 1$.

Theorem~\ref{thm:size_divisibility} leads to the following divisibility condition for $x$.

\begin{corollary}\label{cor:param}
	Let $\mathcal{F}\subseteq\qbinomS{V}{k}$ be a set of degree $1$ and let $x$ be its parameter.
	Then we have
	\[
		\frac{\qnumbS{k}}{\qnumbS{\gcd(n,k)}} \cdot x \in \Z\text{,}
	\]
	restricting the denominator of the fraction $x$ in canceled form.
\end{corollary}

\begin{proof}
	This follows from Theorem~\ref{thm:size_divisibility} and Lemma~\ref{lem:gcd_helper}.
\end{proof}

\begin{example}
	We study the statement of Corollary~\ref{cor:param} in a few concrete cases.
	\begin{enumerate}[(a)]
		\item
		In the case $k \mid n$, we get $x\in \Z$, as already seen in Lemma~\ref{lem:gcd_max}.
		This result is not new, as it follows from \cite[Th.~2.9, prop.~8]{Blokhuis-DeBoeck-DHaeseleer-2019-DCC87[8]:1839-1856}.
		\item
		For coprime $n$ and $k$ (and in particular for $k$ prime and $k \nmid n$) we get
		\[
			(1 + q + q^2 + \ldots + q^{k-1}) \cdot x \in \Z\text{.}
		\]
		\item
		For $k = 4$ and $n\equiv 2\pmod 4$ we get $(1 + q^2)\cdot x\in\Z$ for $n \equiv 2\pmod 4$.
	\end{enumerate}
\end{example}

A further possibility to approach the gcd-expression in Theorem~\ref{thm:size_divisibility} is given by Lemma~\ref{lem:lambda_min}, leading to the following result.
\begin{lemma}
	Let $f : \qbinomS{V}{k} \to \Z$ be a non-zero function of degree $t$.
	Then 
	\[
	    \prod_{\substack{i\in\Z_{>0} \\ t\leq (n \bmod i) < (k\bmod i)}} \Phi_i(q) \mid \# f\text{.}
    \]
\end{lemma}

\section{Dualization}\label{sec:dual}
In this section, $V$ and (if needed) $W$ are again either a set or a finite $\F_q$-vector space of finite rank.
For bijections $\sigma : \qbinomS{V}{k} \to \qbinomS{W}{\ell}$, the induced map $\R^{\qbinomS{V}{k}} \to \R^{\qbinomS{W}{\ell}}$, $f \mapsto (\sigma(K)) \mapsto f(K))$ will again be denoted by $\sigma$.
The induced map $\sigma : \R^{\qbinomS{V}{k}} \to \R^{\qbinomS{W}{\ell}}$ is an isomorphism of $\R$-vector spaces.
In this way, the dualization map ${\perp} : \qbinomS{V}{k} \to\qbinomS{V}{n-k}$ induces the \emph{dualized function} of a function $f : \qbinomS{V}{k} \to \R$ as $f^\perp : \qbinomS{V}{n-k} \to \R$, $K \mapsto f(K^\perp)$.
For a set $F \subseteq \R^{\qbinomS{V}{k}}$ of functions, we define the dual set of functions as the element-wise image $F^\perp = \{f^\perp \mid f\in F\} \subseteq\R^{\qbinomS{V}{n-k}}$.
Note that $(\boldsymbol{x}_U^{(k)})^\perp = \bar{\boldsymbol{x}}_{U^\perp}^{(n-k)}$, $(\bar{\boldsymbol{x}}_U^{(k)})^\perp = \boldsymbol{x}_{U^\perp}^{(n-k)}$ for all $U\in\mathcal{L}(V)$.

Our goal is to show $\deg(f^\perp) = \deg(f)$ (Theorem~\ref{thm:dual_v_space}), and then to find a formula for the weight distribution of the dual function (Theorem~\ref{thm:dual_wt_dist}).
As an intermediate step, we will compute the weight distribution of the dual pencils (Lemma~\ref{lem:weights_x_bar}).

\newcommand{\varj}{j}
\newcommand{\vari}{i}
\newcommand{\vara}{a}
\newcommand{\varb}{b}
\newcommand{\vart}{t}
\newcommand{\vars}{s}

For non-negative integers $\vara, \varb$ and $i,j\in\{0,\ldots,\varb\}$ we define the numbers
	\[
		\alpha^{(\varb,\vara)}_{\vari\varj} = q^{\varj(\vara+\varj-\vari)} \qbinom{\vara+\varb-\vari}{\varb-\varj}{q} \qbinom{\vari}{\varj}{q}\text{.}
	\]
	Moreover, we define $((\varb+1) \times (\varb + 1))$-matrix
	\[
		A^{(\varb,\vara)} = (\alpha^{(\varb,\vara)}_{\vari\varj})_{\vari,\varj\in\{0,\ldots,\varb\}}\text{.}
	\]
	For $\vari < \varj$ we have $\alpha^{(\varb,\vara)}_{\vari\varj} = 0$, and for the diagonal entries we have
	\[
		\alpha^{(\varb,\vara)}_{\vari\vari} = q^{\vara\vari}\qbinomS{\vara+\varb-\vari}{\varb-\vari} \neq 0\text{.}
	\]
	So $A$ is an invertible lower triangular matrix.
	We denote the entries of its inverse by
	\[
		(A^{(\varb,\vara)})^{-1} = (\beta^{(\varb,\vara)}_{\vari\varj})_{\vari,\varj\in\{0,\ldots,\varb\}}\text{.}
	\]
	Moreover, $\alpha^{(\varb,\vara)}_{\vari\varj} = 0$ if $\vari-\varj > \vara$.
	The matrix $A^{(\varb,0)}$ is a $((\varb+1) \times (\varb+1))$ unit matrix.

\begin{lemma}\label{lem:xbar_x_counting}
	Let $k\in\{0,\ldots,n\}$ and $i\in \{0,\ldots,\min(k,n-k)\}$.
	Furthermore, let $J\in\qbinomS{V}{n-i}$ and $K\in\qbinomS{V}{k}$ and define $y = \cork(J + K)$.
	Then for all $z\in\{0,\ldots,i\}$,
	\begin{multline*}
		\#\Big\{I\in\qbinomS{V}{i} \mid I \leq K \text{ and }\rk(I \cap J) = z\Big\} \\
		= q^{(i-z)(k-i+y-z)}\qbinomS{k - i + y}{z}\qbinomS{i-y}{i-z} = \alpha^{(i,k-i)}_{i-y,i-z}\text{.}
	\end{multline*}
\end{lemma}

\begin{proof}
	Clearly
	\[
		\Big\{I\in\qbinomS{V}{i} \mid I \leq K \text{ and }\rk(I \cap J) = z\Big\}
		=
		\Big\{I\in\qbinomS{K}{i} \mid \rk(I \cap (K \cap J)) = z\Big\}\text{.}
	\]
	Application of Lemma \ref{lem:subspace_counting} (with $V \leftarrow K$, $U \leftarrow I$, $A \leftarrow I \cap J$, $B \leftarrow K\cap J$, $n = k$, $u \leftarrow i$, $a \leftarrow \rk(I \cap J) = z$ and $b \leftarrow \rk(K \cap J) = \rk(K) + \rk(J) - \rk(K+J)  = k + (n-i) - (n-y) = k - i + y$) shows that the cardinality of this set is
	\[
		q^{(k-i+y-z)(i-z)} \qbinomS{k-i+y}{z} \qbinomS{i-y}{i-z} = \alpha^{(i,k-i)}_{i-y,i-z}\text{.}
	\]
\end{proof}

\pagebreak

\begin{lemma}\label{lem:dual_pencil}
	Let $k\in\{0,\ldots,n\}$ and $i\in \{0,\ldots,\min(k,n-k)\}$.
	Furthermore, let $J\in\mathcal{L}(V)$ with $\cork J = i$.
	\begin{enumerate}[(a)]
		\item $\bar{\boldsymbol{x}}^{(k)}_J \in \bar{V}^{(k)}_i$.
		\item For all $I\in\qbinomS{V}{i}$, we have $\wt^{(i)}_{\bar{\boldsymbol{x}}^{(k)}_J}(I) = \beta^{(i,k-i)}_{i-\rk(I\cap J),0}$.
	\end{enumerate}
\end{lemma}

\begin{proof}
	We want to show that
	\begin{equation}\label{eq:x_xbar_deg}
		 \bar{\boldsymbol{x}}_J^{(k)} = \sum_{I\in\qbinomS{V}{i}} \beta^{(i,k-i)}_{z-\rk(I\cap J),z-i}\, \boldsymbol{x}_I^{(k)}\text{.}
	\end{equation}
	To this end, fix a $K\in\qbinomS{V}{k}$ and let $y = \cork(J + K)\in\{0,\ldots,i\}$.
	By Lemma~\ref{lem:xbar_x_counting}, the right hand size of~\eqref{eq:x_xbar_deg} evaluated at $K$ equals
	\begin{multline*}
		\sum_{\substack{I\in\qbinomS{V}{i} \\ I \leq K}} \beta^{(i,k-i)}_{z-\rk(I\cap J),z-i}
		= \sum_{z = 0}^n \alpha^{(i,k-i)}_{i-y,i-z} \beta^{(i,k-i)}_{i-z,0}
		= \big(A^{(i,k-i)}\cdot (A^{(i,k-i)})^{-1}\big)_{i-y,0} \\
		= \delta_{y,i}
		= \delta_{\cork(J+K), \cork(J)}
		= \begin{cases}
			1 & \text{if }K \leq J\text{;} \\
			0 & \text{otherwise,}
		\end{cases}
	\end{multline*}
	which equals the left hand side of~\eqref{eq:x_xbar_deg} at $K$.
\end{proof}

\begin{theorem}\label{thm:dual_v_space}
	For all $k\in\{0,\ldots,n\}$ and all $i\in\{0,\ldots,\min(k,n-k)\}$,
	\[
		(\bar{V}^{(k)}_i)^\perp = \bar{V}^{(n-k)}_i\text{.}
	\]
	In other words, $\deg(f^\perp) = \deg(f)$ for all $f : \qbinomS{V}{k} \to \R$.
\end{theorem}

\begin{proof}
	Let $I\in\qbinomS{V}{i}$.
	By Lemma~\ref{lem:dual_pencil},
	\[
		(\boldsymbol{x}_I^{(k)})^\perp
		= \bar{\boldsymbol{x}}_{I^\perp}^{(n-k)}
		\in \bar{V}^{(n-k)}_i\text{.}
	\]
	Therefore,
	\[
		(\bar{V}^{(k)}_i)^\perp
		= \big\langle\big\{(\boldsymbol{x}_I^{(k)})^\perp  \mid I\in\qbinomS{V}{i}\big\}\big\rangle
		\subseteq \bar{V}^{(n-k)}_i\text{.}
	\]
	Since both sides are $\R$-vector spaces of the same dimension $\qbinomS{n}{k} = \qbinomS{n}{n-k}$, we have equality.
\end{proof}

\begin{remark}
	Via Theorem~\ref{thm:alg_geom}, the result $\deg(f^\perp) = \deg(f)$ in Theorem~\ref{thm:dual_v_space} corresponds to the property of the dual design in design theory.
\end{remark}

\begin{lemma}\label{lem:qbinom_sum}
	For all $m,n,k\in\Z$
	\[
		\sum_{i=0}^k (-1)^i q^{\binom{i}{2}} \qbinom{k-m-1-i}{k-i}{q} \qbinom{n}{i}{q} = (-1)^k q^{\binom{k}{2} - km}\qbinom{m+n}{k}{q}\text{.}
	\]
	In the case $m = -n$, the right-hand side reduces to the Kronecker delta $\delta_{k,0}$.
\end{lemma}

\begin{proof}
	By the negation formula~\eqref{eq:qbinom_negation},
	\[
		\qbinom{m}{k-i}{q} = (-1)^{k-i} \frac{1}{q^{(k-i)(-m) + \binom{k-i}{2}}} \qbinom{-m+k-i-1}{k-i}{q}\text{.}
	\]
	Application of this replacement in the $q$-Vandermonde identity~\eqref{eq:q_vandermonde} yields
	\[
		\qbinom{n+m}{k}{q} = \sum_{i=0}^k (-1)^{k-i} \frac{1}{q^{(k-i)(-m) + \binom{k-i}{2}}} \qbinom{-m+k-i-1}{k-i}{q} \qbinom{n}{i}{q} q^{i(m-k+i)}\text{,}
	\]
	which transforms to the formula in the statement.

	For $m = -n$, we have $\qbinom{n+m}{k}{q} = \qbinom{0}{k}{q} = \delta_{k,0}$.
\end{proof}

\begin{lemma}\label{lem:matrix_inverse}
	For $a \neq 0$,
	\[
		\beta^{(\varb,\vara)}_{\vari\varj} =
		\begin{cases}
			\displaystyle\frac{(-1)^{\vari-\varj}}{q^{\vara\vari + \binom{\vari}{2} - \binom{\varj}{2}}} \, \frac{\qnumb{\vara}{q}}{\qnumb{\vara+\vari-\varj}{q}} \, \frac{\qbinom{\vari}{\varj}{q}}{\qbinom{\vara+\varb-\varj}{\varb-\vari}{q}} & \text{if }j \leq i\text{,} \\
			0 & \text{otherwise.}
		\end{cases}
	\]
\end{lemma}

\begin{proof}
	Let $A = A^{(\varb,\vara)}$ and $B$ be the lower triangular $((\varb+1)\times(\varb+1))$-matrix with the entries $\beta^{(b,a)}_{ij}$ as given in the statement.
	As a product of two lower triangle matrices, $BA$ is a lower triangular matrix.
	To compute its entries, let $i,j\in\{0,\ldots,b\}$ with $i \geq j$.
	\begin{align*}
		(BA)_{\vari\varj} & = \sum_{\vart=\varj}^{\min(\vari,\vara+\varj)} \frac{(-1)^{\vari-\vart}}{q^{\vara \vari + \binom{\vari}{2} - \binom{\vart}{2}}} \frac{\qnumbS{\vara}}{\qnumbS{\vara+\vari-\vart}}\, \frac{\qbinomS{\vari}{\vart}}{\qbinomS{\vara+\varb-\vart}{\varb-\vari}} \cdot q^{\varj(\vara+\varj-\vart)} \qbinomS{\vara+\varb-\vart}{\varb-\varj} \qbinomS{\vart}{\varj} \\
		& = q^{\vara(\varj-\vari)+\varj^2 - \binom{\vari}{2}} \frac{\qnumbS{\vari}!\, \qnumbS{\varb-\vari}!}{\qnumbS{\varb-\varj}!\,\qnumbS{\varj}!} \sum_{\vart=\varj}^{\min(\vari,\vara+\varj)} (-1)^{\vari-\vart} q^{\binom{\vart}{2} - \varj\vart} \frac{\qnumbS{\vara}\, \qnumbS{\vara+\vari-1-\vart}!}{\qnumbS{\vari-\vart}!\, \qnumbS{\vara+\varj-\vart}!\, \qnumbS{\vart-\varj}!} \\
		& = (-1)^{\vari-\varj} q^{\vara(\varj-\vari)- \binom{\vari}{2} + \binom{\varj}{2}} \frac{\qbinomS{\varb}{\varj}}{\qbinomS{\varb}{\vari}}\, \sum_{\vart=\varj}^{\vari} (-1)^{\vart-\varj} q^{\binom{\vart-j}{2}}\qbinomS{\vara+\vari-1-\vart}{\vari-\vart} \qbinomS{\vara}{\vart-\varj} \\
		& = (-1)^{\vari-\varj} q^{\vara(\varj-\vari)- \binom{\vari}{2} + \binom{\varj}{2}} \frac{\qbinomS{\varb}{\varj}}{\qbinomS{\varb}{\vari}}\, \sum_{\vars=0}^{\vari-\varj} (-1)^{\vars} q^{\binom{\vars}{2}}\qbinomS{(\vari-\varj)+\vara-1-\vars}{(\vari-\varj)-\vars} \qbinomS{\vara}{\vars}\text{.}
	\end{align*}
	Applying Lemma~\ref{lem:qbinom_sum} (with $m\leftarrow -\vara$, $n\leftarrow \vara$, $k\leftarrow\vari - \varj$, $i\leftarrow\vars$), this expression equals
	\[
		(-1)^{\vari-\varj} q^{\vara(\varj-\vari)- \binom{\vari}{2} + \binom{\varj}{2}} \frac{\qbinomS{\varb}{\varj}}{\qbinomS{\varb}{\vari}} \delta_{\vari-\varj,0} = \delta_{\vari,\varj}\text{.}
	\]
	So $BA$ is the unit matrix and hence $B = A^{-1}$.

	We remark that $a \neq 0$ was needed to ensure $\vara+\vari-1-\vart \geq 0$, such that the expression $\qnumbS{\vara+\vari-1-\vart}!$ is defined.
\end{proof}

\begin{lemma}\label{lem:weights_x_bar}
	Let $i\in\{0,\ldots,k\}$, $J\in\qbinomS{V}{n-i}$, $I\in\qbinomS{V}{i}$ and $z = \rk(I \cap J)$.
	Then
\[
	\wt^{(i)}_{\bar{\boldsymbol{x}}^{(k)}_J}(I)
	= \gamma(k,i,z)
	\coloneqq
	\begin{cases} \delta_{z,k} & \text{if }i = k\text{,}\\
	\displaystyle(-1)^{i-z} \frac{1}{q^{(k-i)(i-z) + \binom{i-z}{2}}} \frac{\qnumbS{k-i}}{\qnumbS{k-z}} \frac{1}{\qbinomS{k}{z}}& \text{otherwise.}
	\end{cases}
\]
\end{lemma}

\begin{proof}
	We have already seen that $\wt^{(i)}_{\bar{\boldsymbol{x}}^{(k)}_J}(I) = \beta^{(i,k-i)}_{i-\rk(I\cap J),0}$.
	For $i = k$, $A^{(k,0)} = (A^{(k,0)})^{-1}$ is a unit matrix.
	For $i < k$, evaluation of the formula in Lemma~\ref{lem:matrix_inverse} concludes the proof.
\end{proof}

\begin{theorem}\label{thm:dual_wt_dist}
	Let $k\in\{0,\ldots,n\}$, $i\in\{\max(\deg(f),0),\ldots\min(k,n-k)\}$ and $f\in \bar{V}^{(k)}_i$.
	Then the $i$-weight distribution of $f^\perp\in \bar{V}^{(n-k)}_i$ is given by
	\[
		\wt^{(i)}_{f^\perp}(J) = \sum_{I\in\qbinomS{V}{i}} \gamma(n-k,i,\rk(I^\perp \cap J))\, \wt^{(i)}_f(I)
	\]
	where $J\in\qbinomS{V}{i}$ and $\gamma$ is defined as in Lemma~\ref{lem:weights_x_bar}.
\end{theorem}

\begin{proof}
	For $I\in\qbinomS{V}{i}$, let $a_I = \wt^{(i)}_f(I)$, so $f = \sum_{J\in\qbinomS{V}{i}} a_I \boldsymbol{x}^{(k)}_I$.
	Dualization and the application of Lemma~\ref{lem:weights_x_bar} gives
	\begin{multline*}
		f^\perp
		= \sum_{I\in\qbinomS{V}{i}} a_I \bar{\boldsymbol{x}}_{I^\perp}^{(n-k)}
		= \sum_{I\in\qbinomS{V}{i}} a_I \sum_{J\in\qbinomS{V}{i}} \gamma(i,n-k,\rk(I^\perp \cap J))\, \boldsymbol{x}_J^{(k)} \\
		= \sum_{J\in\qbinomS{V}{i}} \Big(\sum_{I\in\qbinomS{V}{i}} \wt^{(i)}_f(I)\, \gamma(i,n-k,\rk(I^\perp \cap J))\Big)\, \boldsymbol{x}_J^{(k)}\text{.}
	\end{multline*}
\end{proof}

The weight transformation of Theorem~\ref{thm:dual_wt_dist} can be read as a linear map $\psi^{(i)}_k : \R^{\qbinomS{V}{i}} \to \R^{\qbinomS{V}{i}}$ such that for all $f : \qbinomS{V}{k} \to \R$, $\psi^{(i)}_k(\wt^{(i)}_f) = \wt^{(i)}_{f^\perp}$.
With respect to the standard basis of $\qbinomS{V}{i}$, the matrix representing $\psi^{(i)}_k$ is the symmetric matrix $\Gamma^{(i)}_k\in\R^{\qbinomS{V}{i} \times \qbinomS{V}{i}}$, where the entry at $(J,I)$ is given by $\gamma(i, n-k, \rk(I^\perp \cap J))$.
By $(f^\perp)^\perp = f$, the map $\psi^{(i)}_k$ is invertible with inverse $\psi^{(i)}_{n-k}$.

\pagebreak

We collect a few properties of the map $\psi^{(i)}_k$.
\begin{lemma}\label{lem:psi_prop}
	Let $k\in\{0,\ldots,n\}$.
	\begin{enumerate}[(a)]
		\item\label{lem:psi_prop:dual} For all integers $0 \leq i \leq \min(k,n-k)$,
		\[
			{\perp^{(n-k,k)}} \circ \phi_{\uparrow}^{(k,i)} = \phi_{\uparrow}^{(n-k,i)} \circ \psi^{(i)}_k\text{,}
		\]
		where ${\perp^{(n-k,k)}}$ denotes the dualization map $\qbinomS{V}{k} \to \qbinomS{V}{n-k}$, $f\mapsto f^\perp$.
		\item\label{lem:psi_prop:inverse} For all integers $0 \leq i \leq \min(k,n-k)$, the map $\psi_k^{(i)}$ is invertible with inverse $\psi_{n-k}^{(i)}$.
		\item\label{lem:psi_prop:change_level} For all integers $0 \leq i \leq j \leq \min(k,n-k)$,
		\[
		\qbinomS{n-k-j}{i-j}\, \psi^{(i)}_k \circ \phi^{(ij)}_{\uparrow} = \qbinomS{k-j}{i-j}\, \phi^{(ij)}_{\uparrow} \circ \psi^{(j)}_k\text{.}
		\]
	\end{enumerate}
\end{lemma}

\begin{proof}
Part~\ref{lem:psi_prop:dual} is just the content of Theorem~\ref{thm:dual_wt_dist}.
Part~\ref{lem:psi_prop:inverse} follows from $(f^\perp)^\perp = f$.
For Part~\ref{lem:psi_prop:change_level}, for any $f\in \bar{V}_i^{(k)}$ we have by Theorem~\ref{thm:dual_wt_dist} and Lemma~\ref{lem:vi_chain}
\begin{multline*}
	\Big(\qbinomS{k-j}{i-j}^{-1} \psi^{(i)}_k \circ \phi_{\uparrow}^{(ij)}\Big)(\wt_f^{(j)})
	= \psi^{(i)}_k\Big(\qbinomS{k-j}{i-j}^{-1}\phi_{\uparrow}^{(ij)}(\wt_f^{(j)})\Big)
	= \psi^{(i)}_k(\wt_f^{(i)}) \\
	= \wt_{f^\perp}^{(i)}
	= \qbinomS{n-k-j}{i-j}^{-1} \phi_{\uparrow}^{(ij)}(\wt_{f^\perp}^{(j)})
	= \qbinomS{n-k-j}{i-j}^{-1} \phi_{\uparrow}^{(ij)}(\psi^{(j)}_k(\wt_{f}^{(j)}))
\end{multline*}
\end{proof}

\begin{corollary}\label{cor:anti_isomorphism_deg_wt}
	Let $\tau : \mathcal{L}(V) \to \mathcal{L}(W)$ be an anti-isomorphism of lattices (so $\rk(W) = \rk(V) = n$).
	Let $k\in\{0,\ldots,n\}$ and $f\in\qbinomS{V}{k}$.
	Then $\deg \tau(f) = \deg(f)$ and for all $i\in\{\max(\deg(f),0),\ldots,\min(k,n-k)\}$, the $i$-weight distribution $\wt^{(i)}_{\tau(f)} : \qbinomS{W}{i} \to \R$ is given by%
	\footnote{In the expression $\tau(f)$, the symbol $\tau$ denotes the induced map $\R^{\qbinomS{V}{k}} \to \R^{\qbinomS{W}{n-k}}$, $f \mapsto (\tau(K) \mapsto f(K))$.}
	\[
		\wt^{(i)}_{\tau(f)}(J) = \sum_{I\in\qbinomS{V}{i}} \gamma(n-k,i,\rk_W(\tau(I) \cap J))\, \wt^{(i)}_f(I)
	\]
	where $J\in\qbinomS{W}{i}$.
\end{corollary}

\begin{proof}
	As the dualization map ${\perp} : \mathcal{L}(V) \to\mathcal{L}(V)$ is an anti-isomorphism of lattices and any two anti-isomorphisms are related by a lattice isomorphism, there is an automorphism $\theta : \mathcal{L}(W) \to \mathcal{L}(W)$ with $\tau = \theta \circ {\perp}$.
	Since isomorphisms preserve the degree and are compatible with the weight distribution, the claim follows from Theorems~\ref{thm:dual_v_space} and~\ref{thm:dual_wt_dist}.

	Alternatively, one could repeat proof of Theorems~\ref{thm:dual_v_space} and~\ref{thm:dual_wt_dist} and the preceding lemmas with dualization ${\perp} : \mathcal{L}(V) \to\mathcal{L}(V)$ replaced by the anti-isomorphism $\tau : \mathcal{L}(V) \to\mathcal{L}(W)$, using that $\tau(\boldsymbol{x}_U^{(k)}) = \bar{\boldsymbol{x}}_{\tau(U)}^{(n-k)}$ for all $U\in\mathcal{L}(V)$.
\end{proof}

\section{Change of ambient space}\label{sec:weights_degree:ambient}
There are the following two elementary ways to shrink the ambient space $V$.
Modding out a point $P\in\qbinomS{V}{1}$, and restriction to a hyperplane $H\in\qbinomS{V}{n-1}$.
Note that both operations are dual to each other.
The reverse operations give two elementary ways to enlarge the ambient space $V$.
In the following, we will investigate how degree and weights are affected by these operations.
Application of these modifications to (dual) pencils will give basic examples of sets of a given degree $t$ which we will call of type $\mathcal{F}$.

\begin{lemma}\label{lem:weights_pencil}
	Let $k\in\{0,\ldots,n\}$, $P\in\qbinomS{V}{1}$ be a point and let $f : \qbinomS{V}{k} \to \R$ be a non-zero function of degree $t$ such that $f(K) = 0$ for all $K\in\qbinomS{V}{k}$ with $P \nleq K$.

	If $k \leq \frac{n}{2}$ or $t \neq n-k$, then $\wt^{(t)}_f(T) = 0$ for all $T\in\qbinomS{V}{t}$ with $P \nleq K$.
\end{lemma}

\begin{proof}
	The assumption on $t$ and $k$ is equivalent to $k = t = \frac{n}{2}$ or $k \leq n-t-1$.
	In the first case, $\wt^{(t)}_f(T) = f(T)$ for all $T\in\qbinomS{V}{t} = \qbinomS{V}{k}$, so the statement is true.

	There remain the cases with $k \leq n-t-1$.
	Let
	\[
		\mathcal{T} = \{T\in\qbinomS{V}{t} \mid P \nleq T\}
		\quad\text{and}\quad
		\mathcal{K} = \{K\in\qbinomS{V}{k} \mid P \nleq K\}
	\]
	be the set of all $t$-subspaces not containing $P$, and of all $k$-subspaces not containing $P$, respectively.
	Then $\#\mathcal{T} = \qbinomS{n}{t} - \qbinomS{n-1}{t-1} = q^t \qbinomS{n-1}{t}$.
	We have that $\phi_{\uparrow}^{(kt)}(\wt_f^{(t)}) = f$.
	The partition $\qbinomS{V}{k} = (\qbinomS{V}{k}\setminus \mathcal{K}) \cup \mathcal{K}$ and $\qbinomS{V}{t} = (\qbinomS{V}{t} \setminus \mathcal{T}) \cup \mathcal{T}$ splits the matrix $(W^{(n;t,k)})^\top$ representing $\phi_{\uparrow}^{(kt)}$ into the block matrix
	\[
		(W^{(n;t,k)})^\top
		= \begin{pmatrix}
			(W^{(n-1;t-1,k-1)})^\top & \ast \\
			0 & A^\top
		\end{pmatrix}\text{,}
	\]
	where $A$ is the $\mathcal{T}$-vs.-$\mathcal{K}$ incidence matrix $A$.
	By the assumption on $f$, we get that $(\wt^{(t)}(T))_{T\in\mathcal{T}}$ is in the (right) kernel of $A^\top$.

	For $q = 1$, $A$ is the matrix $W^{(n-1;t,k)}$ of the ambient space $\mathcal{L}(V \setminus\{P\})$, which by $t\leq k \leq (n-1)-t$ is of full rank $\binom{n-1}{t} = q^t \qbinom{n-1}{t}{1}$.
	For $q \geq 2$, the matrix $A$ is the incidence matrix of an attenuated space.
	It has full rank $q^t \qbinomS{n-1}{t}$ by~\cite[Thm.~5]{Guo-Li-Wang-2014-DM315_316:42-46}, which is the number of rows of $A$.%
	\footnote{Denoting the symbols of~\cite[Thm.~5]{Guo-Li-Wang-2014-DM315_316:42-46} by $n'$, $k'$, $d'$ and $\ell'$, we apply the theorem with $n' = n-1$, $k' = k$, $d' = t$ and $\ell' = 1$.
	As stated in~\cite[Thm.~5]{Guo-Li-Wang-2014-DM315_316:42-46}, there is the requirement $1\leq d' < k' \leq n'-d'$, but clearly it is true also for $d' = 0$ (where $A$ is an all-one row vector) and $d' = k'$ (where $A$ is a unit matrix).
	The requirement $k' \leq n'-d'$ is met by $k \leq n-1-t$.
	}
	Hence the kernel of $A^\top$ is trivial.
\end{proof}

\begin{remark}
Lemma~\ref{lem:weights_pencil} is generally not true for values $n,k,t$ not covered by the assumption, i.e.\ for $k > \frac{n}{2}$ with $t = n-k$.
For a counterexample, fix $K\in\qbinomS{V}{k}$ with $P \leq K$ and consider $f = \chi_{\{K\}}$.
Then $\deg f = \deg f^\perp = \deg \boldsymbol{x}^{(n-k)}_{K^\perp} = n-k$.
It is easily checked that for $K'\in\qbinomS{V}{n-k}$,
\[
	\wt^{(n-k)}_f(K') = \begin{cases}\frac{1}{\qbinomS{k}{n-k}} & \text{if }K' \leq K\text{,}\\0 & \text{if } K' \nleq K\end{cases}
\]
is the weight distribution of $f$.
Because of $k > \frac{n}{2}$ we have $n-k < k$, so there exists a $K'\in\qbinomS{K}{n-k}$ with $P \nleq K'$.
However by the above formula, $\wt^{(n-k)}_f(K') \neq 0$.
\end{remark}

Now we are ready to prove the core results on the two elementary modifications of the ambient space.

\begin{theorem}\label{thm:change_ambient_P}
	Let $n \geq 1$, $P\in\qbinomS{V}{1}$ be a point and $k \geq 1$.
	Then
	\[
		\Phi : \R^{\qbinomS{V/P}{k-1}} \to \R^{\qbinomS{V}{k}}\text{,}\qquad
		\Phi(f) : K \mapsto  \begin{cases} f(K/P) & \text{if }P \leq K\text{,}\\0 & \text{if }P\nleq K\end{cases}
	\]
	is an injective $\R$-linear map with image
	\[
		\im(\Phi) = \{g \in \R^{\qbinomS{V}{k}} \mid \supp g \subseteq {\textstyle\qbinomS{V}{k}}|_P\}
	\]
	and the property that for all $f : \qbinomS{V/P}{k-1} \to \R$
	\[
		\deg_{V}(\Phi(f)) = \min(\deg_{V/P}(f) + 1,\; n-k)\text{.}
	\]
	Moreover, for all $e\in\{\max(\deg_{V/P}(f),0)\,+\,1,\ldots,\min(k,n-k)\}$ and $E\in\qbinomS{V}{e}$ we have
	\[
		\wt^{(e)}_{\Phi(f)}(E)
		= \begin{cases}
			\wt^{(e-1)}_f(E/P) & \text{if }P \leq E\text{,}\\
			0 & \text{if }P \nleq E\text{.}
		\end{cases}
	\]
\end{theorem}

\begin{proof}
	It is clear that $\Phi$ is an injective $\R$-linear map with the stated image.

	We fix a map $f : \qbinomS{V/P}{k-1}\to\R$ and define $t' = \deg_{V/P}(f)$ and $t = \deg_V(\Phi(f))$.
	For $f = \vek{0}$ we have $\Phi(f) = \vek{0}$ and the claim is easily checked.
	So in the following $f \neq \vek{0}$.
	Then $\Phi(f)$ is not constant and hence $t \geq 1$.
	For $e\in\{t'+1,\ldots,\min(k,n-k)\}$, it is easily checked that the stated formula for $\wt^{(e)}_{\Phi(f)}(E)$ gives valid $e$-weights for $\Phi(f)$, implying that there exist $(t'+1)$-weights for $\Phi(f)$.
	Hence $t \leq t' + 1$.

	If $k \leq \frac{n}{2}$ or $t \neq n-k$, then by Lemma~\ref{lem:weights_pencil}, $\wt^{(t)}_{\Phi(f)}(T) = 0$ for all $T\in\qbinomS{V}{t}$ with $P \nleq T$.
	Based on that, one can check that $f$ complies with the weight formula $\wt^{(t-1)}_{f}(T/P) = \wt^{(t)}_{\Phi(f)}(T)$ for all $T/P\in \qbinomS{V/P}{t-1}$.
	As before, this implies $t' \leq t - 1$.
	Altogether, $t = t' + 1 = \min(t' + 1, n-k)$.

	In the remaining case $k > \frac{n}{2}$ and $t = n-k$, the inequality $t\leq t' + 1$ yields $t' \geq n - k - 1$, which implies $t = \min(t'+1,n-k)$.
\end{proof}

The following is the dualized version of Theorem~\ref{thm:change_ambient_P}.

\begin{theorem}\label{thm:change_ambient_H}
	Let $n \geq 1$, $H\in\qbinomS{V}{n-1}$ be a hyperplane and $k \leq n-1$.
	Then the map
	\[
		\Psi : \R^{\qbinomS{H}{k}} \to \R^{\qbinomS{V}{k}},\qquad
		\Psi(f) : K \mapsto  \begin{cases} f(K) & \text{if }K \leq H\text{,}\\0 & \text{if }K\nleq H\end{cases}
	\]
	is an injective $\R$-linear map with image
	\[
		\im(\Psi) = \{g \in \R^{\qbinomS{V}{k}} \mid \supp g \subseteq {\textstyle\qbinomS{H}{k}}\}
	\]
	and the property that for all $f : \qbinomS{H}{k} \to \R$
	\[
		\deg_{V}(\Psi(f)) = \min(\deg_H(f) + 1,\; k)\text{.}
	\]
	Moreover, for all $e\in\{\max(\deg_{H}(f),0)\,+\,1,\ldots,\min(k,n-k)\}$ and all $E\in\qbinomS{V}{e}$ we have
	\[
		\wt^{(e)}_{\Psi(f)}(E)
		= \begin{cases}
			\displaystyle\frac{1}{\qnumb{k-e+1}{q}} \sum_{E'\in\qbinomS{E}{e-1}} \wt^{(e-1)}_{f}(E') & \text{if }E \leq H\text{,}\\[8mm]
			\displaystyle-\frac{\qnumb{k-e}{q}}{q^{k-e}\qnumb{k-e+1}{q}} \wt^{(e-1)}_f(E \cap H) & \text{if }E \nleq H\text{.}
		\end{cases}
	\]
\end{theorem}

\begin{proof}
	It is clear that $\Psi$ is an injective $\R$-linear map with the stated image.

	We define the point $P = H^\perp \in \qbinomS{V}{1}$.
	Let $\Phi : \qbinomS{V/P}{n-k-1} \to \qbinomS{V}{n-k}$ be defined as in Theorem~\ref{thm:change_ambient_P} (with $n-k$ in the role of $k$).
	It is straightforward to check that $\Psi = {\perp} \circ \Phi \circ {\operp}$, where $\perp : \R^{\qbinomS{V}{n-k}} \to\R^{\qbinomS{V}{k}}$ denotes the usual dualization and $\operp : \R^{\qbinomS{H}{k}} \to \R^{\qbinomS{V/P}{n-k-1}}$ denotes the bijection induced by the anti-isomorphism $\mathcal{L}(H) \to \mathcal{L}(V/P^\perp)$, $U \mapsto U^\perp / P^\perp$.
	Since the maps ${\perp}$ and ${\operp}$ both preserve the degree by Corollary~\ref{cor:anti_isomorphism_deg_wt}, we have $\deg(\Psi) = \deg(\Phi)$, such that the statement about $\deg(\Psi)$ follows from Theorem~\ref{thm:change_ambient_P}.%
	\footnote{Notice that due to duality $n-k$ is replaced by $k$.}

	Now let $f : \qbinomS{H}{k} \to \R$.
	We determine the function $g\in\R^{\qbinomS{V}{k}}$ which is represented by the weight formula $w : \qbinomS{V}{e} \to \R$ in the statement.
	Let $K\in \qbinomS{V}{k}$.

	We first consider the case $K \leq H$.
	Then for any $E'\in\qbinomS{K}{e-1}$ there are $\qbinomS{k-e+1}{e-(e-1)} = \qnumb{k-e+1}{q}$ subspaces $E\in\qbinomS{K}{e}|_{E'}$.
	We get
	\begin{align*}
		g(K)
		& = \sum_{E\in\qbinomS{K}{e}} w(E)
		= \sum_{E\in\qbinomS{K}{e}} \frac{1}{\qnumb{k-e+1}{q}}\sum_{E'\in\qbinomS{E}{e-1}} \wt^{(e-1)}_f(E') \\
		& = \frac{1}{\qnumb{k-e+1}{q}}\sum_{E'\in\qbinomS{K}{e-1}} \#\qbinomS{K}{e}|_{E'}\,\wt^{(e-1)}_f(E')
		= f(K)\text{.}
	\end{align*}

	In the remaining case $K \nleq H$, the intersection $K' = K \cap H$ is of rank $\rk K' = k-1$.
	Now the $e$-subspaces $E$ of $K$ may either be contained in $H$ (equivalently, in $K'$) or not.
	The set of the former ones is $\qbinomS{K'}{e}$, and each $E'\in\qbinomS{K'}{e-1}$ is contained in $\qbinomS{(k-1)-(e-1)}{e-(e-1)} = \qnumb{k-e}{q}$ subspaces $E$ of the first kind.
	In the latter case, $\rk (E \cap K') = e-1$, and each $E'\in\qbinomS{K'}{e-1}$ arises as $E' = E \cap K'$ for $\qbinomS{k-(e-1)}{e - (e-1)} - \qbinomS{(k-1)-(e-1)}{e - (e-1)} = q^{k-e}$ subspaces $E$ of the second kind.
	So
	\begin{align*}
		g(K)
		& = \sum_{E\in\qbinomS{K'}{e}} w(E) + \sum_{E\in\qbinomS{K}{e}\setminus\qbinomS{K'}{e}} w(E) \\
		& = \frac{1}{\qnumb{k-e+1}{q}}\Big(\sum_{E\in\qbinomS{K'}{e}}\big(\sum_{E'\in\qbinomS{E}{e-1}} \wt^{(e-1)}_f(E')\big) - \sum_{E\in\qbinomS{K}{e}\setminus\qbinomS{K'}{e}} \frac{\qnumb{k-e}{q}}{q^{k-e}}\, \wt^{(e-1)}_f(E \cap H)\Big) \\
		& = \frac{1}{\qnumb{k-e+1}{q}}\Big(\sum_{E'\in\qbinomS{K'}{e-1}}\qnumb{k-e}{q} \wt^{(e-1)}_f(E') - \sum_{E'\in\qbinomS{K'}{e-1}}q^{k-e}\, \frac{\qnumb{k-e}{q}}{q^{k-e}}\wt^{(e-1)}_f(E')\Big) = 0\text{.}
	\end{align*}
	Hence $g = \Psi(f)$ and thus $w = \wt_{\Psi(f)}^{(e)}$.
\end{proof}

\begin{remark}
	In the above proof, the weight distribution of $\Psi$ has been verified by a direct computation.
	Alternatively, we could have proceeded like for the degree using the representation $\Psi = {\perp} \circ \Phi \circ {\operp}$ and applying the weight transformation formula of Corollary~\ref{cor:anti_isomorphism_deg_wt} twice.
\end{remark}

\begin{definition}
	Let $k\in\{0,\ldots,n\}$ and $I \leq J \leq V$.
	We define the \emph{basic set}
	\[
		\mathcal{F}^{(V,k)}(I,J) = \mathcal{F}^{(k)}_{I,J} = \qbinomS{V}{k}|^{J}_I\text{,}
	\]
	which is the set of all $k$-subspaces contained in the interval $[I,J]$ in the lattice $\mathcal{L}(V)$.
	They generalize the pencils and dual pencils, which arise as the characteristic functions of $\mathcal{F}^{(V,k)}(I,J)$ in the case $I = \boldsymbol{0}$ and $J = V$, respectively.
	Its size is $\#\mathcal{F}^{(V,k)}(I,J) = \qbinomS{n-j-i}{k-i}$.
	Any to sets with the same values $i = \rk I$, $j = \cork J$ (and the same $n$ and $k$) are isomorphic.
	They will be called sets of \emph{type} $\mathcal{F}^{(n,k)}_{i,j}$.

	The dual of a set of type $\mathcal{F}^{(n,k)}_{i,j}$ is of type $\mathcal{F}^{(n,n-k)}_{j,i}$, namely $(\mathcal{F}^{(V,k)}_{I,J})^\perp = \mathcal{F}^{(V,n-k)}_{J^\perp,I^\perp}$.
	Generally, the supplement of a basic set is not a basic set, with the exception of $(\mathcal{F}^{(V,k)}_{I,J})^\complement = \mathcal{F}^{(V,k)}_{J^\complement,I^\complement}$ in the case $q = 1$ and $\#I + \#J = 1$.
\end{definition}

\begin{theorem}\label{thm:deg_f}
	Let $k\in\{0,\ldots,n\}$ and $I \leq J \leq V$.
	Let $i = \rk I$ and $j = \cork J$.
	Then
	\[
		\deg \mathcal{F}^{(V,k)}(I,J)
		= \begin{cases}
			-\infty & \text{if } i > k\text{ or } j > n-k\text{,} \\
			\min(i+j,k,n-k) & \text{otherwise.}
		\end{cases}
	\]
\end{theorem}

\begin{proof}
	If $i > k$ or $j > n-k$, then $\mathcal{F}^{(k)}(I,J) = \emptyset$ is of degree $-\infty$.
	For the remaining cases, we apply induction over $i + j$.
	If $i = j = 0$, then $\mathcal{F}^{(k)}(I,J) = \qbinomS{V}{k}$ is of degree $0$.
	Let $g : \qbinomS{V}{k}  \to \R$ be the characteristic function of $\mathcal{F}^{(k)}(I,J)$.

	If $i \geq 0$, let $P \in\qbinomS{I}{1}$.
	In the notation of Theorem~\ref{thm:change_ambient_P}, we have $g\in\im(\Phi)$, and $f = \Phi^{-1}(g)\in\R^{\qbinomS{V/P}{k-1}}$ is of type $\mathcal{F}^{(n-1,k-1)}_{i-1,j}$.
	By induction, $\deg_{V/P}(f) = \min(i+j-1,k-1,n-k)$.
	Now by Theorem~\ref{thm:change_ambient_P}
	\begin{multline*}
		\deg\mathcal{F}^{(k)}(I,J)
		= \deg_V \Phi(f)
		= \min(\deg_{V/P}(f) + 1, n-k) \\
		= \min(i+j,k,n-k+1,n-k)
		= \min(i+j,k,n-k)\text{.}
	\end{multline*}

	If $j \geq 0$, let $H \in\qbinomS{J}{n-j-1}$.
	In the notation of Theorem~\ref{thm:change_ambient_H}, we have $g\in\im(\Psi)$, and $f = \Psi^{-1}(g)\in\R^{\qbinomS{H}{k}}$ is of type $\mathcal{F}^{(n-1,k)}_{i,j-1}$.
	By induction, $\deg_H(f) = \min(i+j-1,k,n-k-1)$.
	Now by Theorem~\ref{thm:change_ambient_P}
	\begin{multline*}
		\deg\mathcal{F}^{(k)}(I,J)
		= \deg_V \Psi(f)
		= \min(\deg_H(f) + 1, k) \\
		= \min(i+j,k+1,n-k,k)
		= \min(i+j,k,n-k)\text{.}
	\end{multline*}
\end{proof}

In particular for $t\in\{0,\ldots,\min(k,n-k)\}$, all sets type $\mathcal{F}^{(n,k)}_{i,j}$ with $i + j = t$ are of degree $t$.
We will refer to these as the basic sets or the basic Boolean functions of degree~$t$.

\begin{remark}
	Via Theorem~\ref{thm:alg_geom}, the statement of Theorem~\ref{thm:deg_f} corresponds to the property of the $i$-fold derived and the $j$-fold residual design, see \cite{Kiermaier-Laue-2015-AiMoC9[1]:105-115} for the subspace designs.
\end{remark}

The following corollary summarizes the degrees of the pencils and their duals, which are covered by Theorem~\ref{thm:deg_f} as the special case $I = \boldsymbol{0}$ and $J = V$, respectively.

\begin{corollary}\label{cor:deg_x_xbar}
	Let $k\in\{0,\ldots,n\}$ and $I,J\in\mathcal{L}(V)$ with $i = \rk I$ and $j = \cork J$.
	Then
	\begin{align*}
		\deg \boldsymbol{x}^{(k)}_I
		& = \begin{cases}
		-\infty & \text{if } i > k\text{,} \\
		\min(i,n-k) & \text{if }i \leq k\text{,}
		\end{cases}
		\\
		\text{and}\quad 
		\deg\bar{\boldsymbol{x}}^{(k)}_J
		& = \begin{cases}
		-\infty & \text{if } j > n-k\text{,} \\
		\min(j,k) & \text{if }j \leq n-k\text{.}
		\end{cases}
	\end{align*}
\end{corollary}

Now we can give the missing proof of Lemma~\ref{lem:deg_operations}\ref{lem:deg_operations:prod}.
\begin{proof}[Lemma~\ref{lem:deg_operations}\ref{lem:deg_operations:prod}]
	We may assume that both $f$ and $g$ are non-zero.
	With $i = \deg f$ and $j = \deg g$, there are weights $a_I, b_J\in\R$ such that $f = \sum_{I\in\qbinomS{V}{i}} a_I \boldsymbol{x}^{(k)}_I$ and $g = \sum_{J\in\qbinomS{V}{j}} b_J \boldsymbol{x}^{(k)}_J$.
	Then
	\[
		fg = \Big(\sum_{I\in\qbinomS{V}{i}} a_I \boldsymbol{x}^{(k)}_I\Big)\Big(\sum_{J\in\qbinomS{V}{j}} b_J \boldsymbol{x}^{(k)}_J\Big)
		= \sum_{I\in\qbinomS{V}{i}}\sum_{J\in\qbinomS{V}{j}} a_I b_J \boldsymbol{x}^{(k)}_I \boldsymbol{x}^{(k)}_J
		= \sum_{I\in\qbinomS{V}{i}}\sum_{J\in\qbinomS{V}{j}} a_I b_J \boldsymbol{x}^{(k)}_{I + J}\text{.}
	\]
	Therefore, the degree is upper bounded by the values $\deg\boldsymbol{x}^{(k)}_{I + J} \leq \rk(I + J) \leq i + j$, where we used the degree formula in Corollary~\ref{cor:deg_x_xbar}.
\end{proof}

The following definition aligns with established concepts in design theory (see \cite{Kiermaier-Laue-2015-AiMoC9[1]:105-115} for subspace designs).

\begin{definition}\label{def:der_res}
	Let $n \geq 1$ and $f : \qbinomS{V}{k}\to \R$.
	For a point $P\in\qbinomS{V}{1}$ and $k \geq 1$, the \emph{derived} function of $f$ in $P$ is
	\[
		\Der_P(f) : \qbinomS{V/P}{k-1} \to \R,\quad K/P \mapsto f(K)\text{.}
	\]
	For a hyperplane $H\in\qbinomS{V}{k-1}$ and $n-k \geq 1$, the \emph{residual} function of $f$ in $H$ is
	\[
		\Res_H(f) : \qbinomS{H}{k} \to \R, \quad K \mapsto f(K)\text{.}
	\]
\end{definition}

\begin{corollary}\label{cor:der_res_deg}
	In the situation of Definition~\ref{def:der_res},
	\[
		\deg_{V/P}(\Der_P(f)) \leq \deg_V(f)
		\quad\text{and}\quad
		\deg_H(\Res_H(f)) \leq \deg_V(f)\text{.}
	\]
\end{corollary}

\begin{proof}
	We have $\Der_P(f) = \Phi^{-1}(\vek{x}_P \cdot f)$ and $\Res_H(f) = \Psi^{-1}(\bar{\vek{x}}_H \cdot f)$.
	Now apply Lemma~\ref{lem:deg_operations}\ref{lem:deg_operations:prod} and Theorem~\ref{thm:change_ambient_P} and~\ref{thm:change_ambient_H}, resp.
\end{proof}

Corollary~\ref{cor:der_res_deg} is essentially~\cite[Lem.~6(b),(c)]{DeBeule-DHaeseleer-Ihringer-Mannaert-2023-ElecJComb30[1]:P1.31}, see also~\cite[Th.~4.1]{DeBeule-Mannaert-Storme-2024-DCC-prelim} and~\cite[Th.~3.1]{DeBeule-Mannaert-Storme-2022-DCC90[3]:633-651}.

\section{An example}\label{sec:example}
We investigate the Johnson scheme $J(6,3)$, meaning that we look at the case $q=1$, $V = \{1,2,3,4,5,6\}$, $n=6$ and $k=3$.
The chain of subspaces from Lemma~\ref{lem:vi_chain} is
\[
	\langle\vek{1}\rangle = \bar{V}_0 \subseteq \bar{V}_1 \subseteq \bar{V}_2 \subseteq \bar{V}_3 = \R^{\binom{V}{3}}
\]
of dimensions
\begin{multline*}
	\dim\bar{V}_0 = \binom{6}{0} = 1\text{,}\quad
	\dim\bar{V}_1 = \binom{6}{1} = 6\text{,}\quad \\
	\dim\bar{V}_2 = \binom{6}{2} = 15\quad\text{and}\quad
	\dim\bar{V}_3 = \binom{6}{3} = 20\text{.}
\end{multline*}
The corresponding orthogonal decomposition of $\R^{\binom{V}{3}}$ as discussed in Remark~\ref{rem:vi_vs_johnson} is
\[
	\R^{\binom{V}{3}} = V_0 \perp V_1 \perp V_2 \perp V_3
\]
with dimensions
\[
	 \dim V_0= 1\text{,}\quad
	 \dim V_1= 5\text{,}\quad
	 \dim V_2= 9\quad\text{and}\quad
	 \dim V_3=5\text{.}
\]
In the following, we derive bases of the spaces $V_0 \perp V_i$ with $i\in\{1,2,3\}$ with a rich structure.
Together with $V_0 = \langle\vek{1}\rangle$, this gives a good understanding of the eigenspaces $V_i$.

For the space $V_0 \perp V_1 = \bar{V}_1$ of dimension $6$, a natural basis is given by the set of the $6$ point pencils.

For the other two cases, we consider the set
\[
	\mathcal{F} = \{123,456\}\in\binom{V}{3}\text{.}
\]
By the classification result stated in Remark~\ref{rem:m_q} (or by a short direct computation), $\mathcal{F}$ is not of degree $1$.
It is straightforward to check that $\mathcal{F}$ has the $2$-weights (where $T \in\binom{V}{2}$)
\[
	\wt_{\mathcal{F}}^{(2)}(T)
	= \begin{cases}
		\frac{1}{3} & \text{if }T \subseteq \{1,2,3\} \text{ or }T \subseteq \{4,5,6\}\text{,} \\
		-\frac{1}{6} & \text{otherwise.}
	\end{cases}
\]
Hence $\deg \mathcal{F} = 2$ and therefore $\chi_{\mathcal{F}} \in \bar{V}_2 = V_0 \perp V_1 \perp V_2$.

As a partition of $V$ into parts of size~$3$, the set $\mathcal{F}$ is also a $1$-$(6,3,1)$ design and hence $\chi_{\mathcal{F}} \in U_{t,{\ast}} = \bar{V}_1^\perp + \langle\vek{1}\rangle = V_0 \perp V_2 \perp V_3$.
Altogether,
\[
	\chi_{\mathcal{F}} \in (V_0 \perp V_1 \perp V_2) \cap (V_0 \perp V_2\perp V_3) = V_0 \perp V_2\text{.}
\]
Clearly, the same is true for all $\frac{1}{2}\binom{6}{3} = 10$ isomorphic copies of $\mathcal{F}$.
Indeed, the resulting characteristic functions turn out to be linearly independent and hence span the full space $V_0 \perp V_2$ of dimension $1 + 9 = 10$.

The geometric property of $\mathcal{F}$ yields for each real $2$-$(6,3,\lambda)$ design $\delta$ that%
\footnote{We have $\lambda_{\max} = \binom{n-t}{k-t} = \binom{4}{1} = 4$.}
\[
	\langle \chi_{\mathcal{F}}, \delta\rangle
	= \frac{\lambda}{\lambda_{\max}} \cdot \#\mathcal{F}
	= \frac{\lambda}{2}\text{.}
\]
Hence, each simple $2$-$(6,3,2)$ design contains exactly one of the blocks $\{1,2,3\},\{4,5,6\}$, and by the same argument applied to the isomorphic copies of $\mathcal{F}$, any simple $2$-$(6,3,2)$ design is \emph{anti-complementary}, meaning that it contains exactly one element from each complementary pair of $3$-subsets of $V$.
We mention that this property can also be derived using the Mendelsohn equations on the intersection numbers from design theory, see for example \cite{Kiermaier-Pavcevic-2015-JCD23[11]:463-480} for an overview.

It is easily checked that
\[
	D = \{123,124,136,145,156,246,256,235,345,346\}
\]
is an example of a $2$-$(6,3,2)$ design, and indeed it is anti-complementary.
Hence
\[
	\chi_D \in U_{2,{\ast}} = \bar{V}_2^\perp + \langle \vek{1}\rangle = V_0 \perp V_3\text{.}
\]
The automorphism group of $D$ in the symmetric group $S_V$ is isomorphic to the alternating group $A_5$.
Hence there are $\#S_V / \#A_5 = 12$ isomorphic copies of $D$, falling into $6$ supplementary pairs.
It turns out that each set of $6$ representatives of these pairs is a basis of the full space $V_0 \perp V_3$ of dimension $1 + 5 = 6$.
The set of representatives where each design contains the block $\{1,2,3\}$ is given by
\begin{align*}
     \{123, 124, 135, 146, 156, 236, 245, 256, 345, 346\}\text{,} \\
     \{123, 124, 136, 145, 156, 235, 246, 256, 345, 346\}\text{,} \\
     \{123, 125, 134, 146, 156, 236, 245, 246, 345, 356\}\text{,} \\
     \{123, 125, 136, 145, 146, 234, 246, 256, 345, 356\}\text{,} \\
     \{123, 126, 134, 145, 156, 235, 245, 246, 346, 356\}\text{,} \\
     \{123, 126, 135, 145, 146, 234, 245, 256, 346, 356\}\text{.}
\end{align*}

\section*{Acknowledgements}
The research of the authors was supported by the Research Network Coding Theory and Cryptography (W0.010.17N) of the Research Foundation – Flanders (FWO).

We would like to thank Lukas Klawuhn for making us aware of the article~\cite{Roos-1982-DelftProgrRep7[2]:98-109}.
We are grateful for the anonymous referee comments which helped to improve the paper.

\emergencystretch=1em
\printbibliography

\end{document}